\documentclass[12pt]{amsart}

\usepackage{amsmath}
\usepackage{amsthm}
\usepackage{amssymb}
\usepackage{caption}
\usepackage[curve]{xypic}
\usepackage[all]{xy}
\usepackage{cite}
\usepackage{enumitem}
\usepackage{xcolor}
\usepackage{color}
\usepackage{url}
\usepackage{graphicx}
\usepackage{tikz}
\usepackage{tikz-cd}
\usepackage{hyperref} 
\usepackage{verbatim}
\usepackage{multirow}

\usepackage{algorithm}
\usepackage{algorithmic}

\setlist{itemsep=4pt, topsep=4pt}

\makeatletter 
\def\@cite#1#2{{\m@th\upshape\bfseries%
[{#1\if@tempswa{\m@th\upshape\mdseries, #2}\fi}]}}
\makeatother 

\theoremstyle{plain}
\newtheorem{thm}{Theorem}[section]
\newtheorem{cor}[thm]{Corollary}

\newtheorem{prop}[thm]{Proposition}
\newtheorem{lem}[thm]{Lemma}

\theoremstyle{definition}

\newtheorem{construction}[thm]{Construction}
\theoremstyle{remark}
\newtheorem{rem}[thm]{Remark}

\numberwithin{equation}{subsection}
\newcommand{\nc}{\newcommand}
\newcommand{\rnc}{\renewcommand}

\nc\bA{\mathbb{A}}
\nc\bB{\mathbb{B}}
\nc\bC{\mathbb{C}}
\nc\bD{\mathbb{D}}
\nc\bE{\mathbb{E}}
\nc\bF{\mathbb{F}}
\nc\bG{\mathbb{G}}
\nc\bH{\mathbb{H}}
\nc\bI{\mathbb{I}}
\nc{\bJ}{\mathbb{J}} 
\nc\bK{\mathbb{K}}
\nc\bL{\mathbb{L}}
\nc\bM{\mathbb{M}}
\nc\bN{\mathbb{N}}
\nc\bO{\mathbb{O}}
\nc\bP{\mathbb{P}}
\nc\bQ{\mathbb{Q}}
\nc\bR{\mathbb{R}}
\nc\bS{\mathbb{S}}
\nc\bT{\mathbb{T}}
\nc\bU{\mathbb{U}}
\nc\bV{\mathbb{V}}
\nc\bW{\mathbb{W}}
\nc\bY{\mathbb{Y}}
\nc\bX{\mathbb{X}}
\nc\bZ{\mathbb{Z}}
\nc\cA{\mathcal{A}}
\nc\cB{\mathcal{B}}
\nc\cC{\mathcal{C}}
\rnc\cD{\mathcal{D}}
\nc\cE{\mathcal{E}}
\nc\cF{\mathcal{F}}
\nc\cG{\mathcal{G}}
\rnc\cH{\mathcal{H}}
\nc\cI{\mathcal{I}}
\nc{\cJ}{\mathcal{J}} 
\nc\cK{\mathcal{K}}
\rnc\cL{\mathcal{L}}
\nc\cM{\mathcal{M}}
\nc\cN{\mathcal{N}}
\nc\cO{\mathcal{O}}
\nc\cP{\mathcal{P}}
\nc\cQ{\mathcal{Q}}
\rnc\cR{\mathcal{R}}
\nc\cS{\mathcal{S}}
\nc\cT{\mathcal{T}}
\nc\cU{\mathcal{U}}
\nc\cV{\mathcal{V}}
\nc\cW{\mathcal{W}}
\nc\cY{\mathcal{Y}}
\nc\cX{\mathcal{X}}
\nc\cZ{\mathcal{Z}}
\nc\bfA{\mathbf{A}}
\nc\bfB{\mathbf{B}}
\nc\bfC{\mathbf{C}}
\nc\bfD{\mathbf{D}}
\nc\bfE{\mathbf{E}}
\nc\bfF{\mathbf{F}}
\nc\bfG{\mathbf{G}}
\nc\bfH{\mathbf{H}}
\nc\bfI{\mathbf{I}}
\nc{\bfJ}{\mathbf{J}} 
\nc\bfK{\mathbf{K}}
\nc\bfL{\mathbf{L}}
\nc\bfM{\mathbf{M}}
\nc\bfN{\mathbf{N}}
\nc\bfO{\mathbf{O}}
\nc\bfP{\mathbf{P}}
\nc\bfQ{\mathbf{Q}}
\nc\bfR{\mathbf{R}}
\nc\bfS{\mathbf{S}}
\nc\bfT{\mathbf{T}}
\nc\bfU{\mathbf{U}}
\nc\bfV{\mathbf{V}}
\nc\bfW{\mathbf{W}}
\nc\bfY{\mathbf{Y}}
\nc\bfX{\mathbf{X}}
\nc\bfZ{\mathbf{Z}}

\newcommand{\xra}{\mathop{\longrightarrow}^}

\nc{\dmo}{\DeclareMathOperator}
\nc{\wt}{\widetilde}
\rnc{\Re}{\operatorname{Re}}
\rnc{\Im}{\operatorname{Im}}
\dmo{\rank}{rank}
\dmo{\End}{End}
\dmo{\Hom}{Hom}
\dmo{\Jac}{Jac}
\dmo{\Id}{Id}
\dmo{\Ann}{Ann}
\dmo{\Area}{Area}
\dmo{\CP}{\bC P^1}
\dmo{\rk}{rk}
\dmo{\rel}{rel}
\dmo{\ra}{\rightarrow}
\dmo{\AGL}{\mathrm{AGL}}
\dmo{\AO}{\mathrm{AO}}
\dmo{\Sym}{\mathrm{Sym}}
\dmo{\Hur}{\mathrm{Hur}}

\nc{\GL}{\mathrm{GL}^+(2, \bR)}

\dmo{\Mod}{Mod}
\dmo{\Aff}{Aff}
\nc{\pair}[1]{\langle #1 \rangle}
\nc{\red}[1]{\textcolor{red}{#1}}
\nc{\blue}[1]{\textcolor{blue}{#1}}
\nc{\para}[1]{\noindent\textbf{#1.}}
\dmo{\Dil}{Dil}

\title[Measures on strata of twisted $1$-forms]{Invariant measures on moduli spaces of twisted holomorphic $1$-forms and strata of dilation surfaces}
\author[Apisa]{Paul~Apisa}
\author[Salter]{Nick~Salter}

\begin{document}
\begin{abstract}
The moduli space of twisted holomorphic $1$-forms on Riemann surfaces, equivalently dilation surfaces with scaling, admits a stratification and $\mathrm{GL}(2, \bR)$-action as in the case of moduli spaces of translation surfaces.  We produce an analogue of Masur-Veech measure, i.e. an $\mathrm{SL}(2, \bR)$-invariant Lebesgue class measure on strata or explicit covers thereof. This relies on a novel computation of cohomology with coefficients for the mapping class group. The computation produces a framed mapping class group invariant measure on representation varieties that naturally appear as the codomains of the periods maps that coordinatize strata.
\end{abstract}

\maketitle


\section{Introduction}

Let $\Sigma_{g}$ be a connected closed orientable genus $g$ surface with a set $C$ of $n \geq 1$ marked points, called \emph{cone points}. Set $\Sigma_{g,n} := \Sigma_g \setminus C$. The \emph{dilation group} $\Dil$ is the group of self-maps of $\bC$ of the form $az+b$ where $a \in \bR_{>0}$ and $b \in \bC$. A \emph{dilation surface structure} on $\Sigma_{g,n}$ is a branched $(\Dil, \bC)$-structure. This is an at atlas of charts to $\bC$ with transition functions in $\mathrm{Dil}$ subject to ``branching conditions" at the punctures. The branching conditions prescribe local geometric models of neighborhoods of the punctures. Formally, let $X$ be the Riemann surface structure that the dilation structure induces on $\Sigma_g$ and let $\nabla: \Omega_{X \setminus C} \ra \Omega_{X \setminus C}^{\otimes 2}$ be the connection that sends a holomorphic $1$-form, written in charts as $f(z) dz$, to $f'(z) dz^{\otimes 2}$, where $\Omega_{X \setminus C}$ is the sheaf of holomorphic $1$-forms. The branching conditions are equivalent to demanding that $\nabla$ extends to a connection $\Omega_X \ra \Omega_X(C) \otimes \Omega_X$. This connection defines a \emph{holonomy homomorphism} $\rho_{hol}: \pi_1(\Sigma_{g,n}) \ra \bR_{>0}$. 

The \emph{local holonomy} of $c \in C$ is $\rho_{hol}(\delta_c)$ where $\delta_c$ is a small positively oriented loop around $c$. Together with the cone angle around $c$, this determines the \emph{signature} $\kappa_c$ of $c$. Setting $\kappa = (\kappa_c)_{c \in C}$, $\cD(\kappa)$ is the \emph{stratum of dilation surfaces} with cone points of prescribed signature. A cone point is an \emph{integral zero} (or \emph{translation-like}) if it admits a local model that ``looks like" a neighborhood of a cone point on a finite area translation surface. Formally, $c \in C$ admits a neighborhood with a nonzero $\nabla$-flat holomorphic $1$-form. The germ of such a $1$-form at $c$ turns a dilation surface into a \emph{scaled dilation surface} when the horizontal foliation of the $1$-form agrees with that of the dilation surface. Let $\Omega^{tw} \cM_{g,n}(\kappa)$ be the collection of scaled dilation surfaces. This is an $\bR_{>0}$-bundle over $\cD(\kappa)$. In the sequel, we will always suppose that dilation surfaces have at least one integral zero. This can be arranged by marking a point if necessary. 

Recall that $\Omega \cM_g(\kappa)$ is the moduli space of pairs $(X, \omega)$ where $X$ is a Riemann surface of genus $g$ and $\omega$ is a holomorphic $1$-form with zeros of order $\kappa$. These elements can be identified with {\em translation surfaces}, from which $\Omega \cM_g(\kappa)$ inherits a $\mathrm{GL}(2, \bR)$ action. Likewise, elements of $\Omega^{tw} \cM_{g,n}(\kappa)$ can be identified with \emph{(real) twisted holomorphic $1$-forms} and their moduli space inherits a $\mathrm{GL}(2, \bR)$ action. Given a Riemann surface structure $Y$ on $\Sigma_{g,n}$ and a character $\chi: \pi_1(\Sigma_{g,n}) \ra \bR_{>0}$, a twisted holomorphic $1$-form on $Y$ is a holomorphic section of $K_Y \otimes L_\chi$ that is meromorphic at $C$, where $K_Y$ is the canonical bundle and $L_\chi$ is the flat line bundle on $Y$ determined by $\chi$. This section is determined by a holomorphic $1$-form $\omega$ on the universal cover of $Y$ so that $\gamma^* \omega = \chi(\gamma)^{-1} \omega$ for all $\gamma$ in $\pi_1(\Sigma_{g,n})$. A twisted holomorphic $1$-form $\omega$ belongs to $\Omega^{tw} \cM_{g,n}(\kappa)$ if, for all $c \in C$, $\chi(\delta_c)$ agrees with the local holonomy of $c$, as specified by $\kappa_c$, and if the order of vanishing of $\omega$ at $c$ agrees with the cone angle (see \cite{ABW1} for a formal definition). 

Of critical importance to studying the dynamics of $\mathrm{GL}(2, \bR)$ on $\Omega \cM_g(\kappa)$ is the existence of \emph{Masur-Veech measure}, an $\mathrm{SL}(2, \bR)$ ergodic invariant Lebesgue class measure constructed by Masur \cite{Ma2} and Veech \cite{V2}. The goal of this paper is to produce an analogue of Masur-Veech measure for $\Omega^{tw} \cM_{g,n}(\kappa)$. The question of whether or not such a measure exists was posed by Ghazouani \cite[Question 1]{Ghazouani-DilationTori} and Ghazouani-Boulanger\footnote{The authors write ``We believe the existence of such a measure to be of capital importance
as it would make the aforementioned analogy [between the dynamics of $\mathrm{SL}(2, \bR)$ on $\Omega^{tw} \cM_{g,n}(\kappa)$ and on infinite volume hyperbolic $3$-manifolds] robust enough to prove interesting theorems about the dynamics of the Teichm\"uller flow".} \cite{BoulangerGhazouani}. Our main result is the following. \vspace{-3mm}

\noindent \textbf{Theorem 1.1 (Paraphrase).} \emph{Up to taking an explicit cover, $\Omega^{tw} \cM_{g,n}(\kappa)$ admits an $\mathrm{SL}(2, \bR)$-invariant fully supported Lebesgue class measure.} \vspace{2mm}

To specify the cover, we must discuss the topology of $\Omega^{tw} \cM_{g,n}(\kappa)$. For any twisted $1$-form $\omega$, $\lambda \omega$ belongs to the same stratum for $\lambda \in \bC^\times$. So $\Omega^{tw} \cM_{g,n}(\kappa)$ is a $\bC^\times$-torsor. By \cite{ABW1}, the quotient $\Omega^{tw} \cM_{g,n}(\kappa)/\bC^\times$ is the framed mapping class group cover of $\cM_{g,n}$. 

Recall that a {\em framing} of $\Sigma_{g,n}$ is a trivialization of the tangent bundle. As discussed in Section \ref{S:MCG}, up to isotopy, a framing is encoded by a cohomology class $W \in H^1(UT \Sigma_{g,n},\bZ)$ which evaluates to $1$ on the class of the fiber, where $UT\Sigma_{g,n}$ is the unit tangent bundle of $\Sigma_{g,n}$. An {\em $N$-framing} for $N \in \bZ_{>0}$ is an element of $H^1(UT\Sigma_{g,n}; \bZ)$ where the homology class of the fiber is sent to $N$ instead of $1$. Given an $N$-framing $\xi$, the subgroup $\mathrm{Mod}_{g,n}[\xi]$ of the mapping class group $\mathrm{Mod}_{g,n}$ that preserves $\xi$ is called the \emph{$N$-framed mapping class group}. When $N = 1$, we just call it the \emph{framed mapping class group}.  

Since scaled dilation surfaces correspond to twisted $1$-forms, a scaled dilation surface specifies an untwisted holomorphic $1$-form in the neighborhood of any \emph{integral cone point}, i.e. one with trivial local holonomy. One may therefore refer to such points as ``zeros" or ``poles" and, in the latter case, to whether the pole is residueless or not. A \emph{restricted stratum} $\Omega^{tw} \cM_{g,n}(\kappa)$ is the subset of a stratum where each integral pole is prescribed to be either residueless or not. We add this data to the signature $\kappa$. Restricted strata remain $\mathrm{GL}(2, \bR)$-invariant. The main result is the following.

\begin{thm}\label{T:main:general}
An explicit cover of a restricted stratum of $\Omega^{tw} \cM_{g,n}$ with at least one integral zero admits a fully supported $\mathrm{SL}(2, \bR)$ invariant Lebesgue class measure on the complement of the locus of translation and homothety surfaces.

The cover corresponds to the subgroup of the framed mapping class group that does the following:
\begin{enumerate}
    \item It preserves the relative homology classes of a fixed tree of arcs connecting integral zeros to one another.
    \item It preserves the relative homology classes of a fixed set of arcs connecting a fixed integral zero to $S$, the set of non-integral cone points and integral poles with nonzero residue. 
    \item It preserves the winding numbers of the arcs in a fixed tree of arcs connecting points in $S$ to one another for some relative framing $\xi$.
\end{enumerate}
\end{thm}

A \emph{relative framing} is an enhancement of a framing that permits the measurement of winding numbers of arcs as well as simple closed curves. See Section \ref{SS:relwnf}. A \emph{homothety surface} is one whose monodromy representation can be conjugated into the homothety subgroup $\{az\}_{a \in \bR_{>0}} \subseteq \Dil$. Note that invariant measures exist on $\Omega^{tw}\cM_{g,1}(2g-2)$ without passing to a cover and, when all cone points are translation-like, one must only pass to a cover that fixes a tree of arcs between the cone points. More generally, we have the following.

\begin{cor}\label{C:main}
If elements of $\Omega^{tw} \cM_{g,n}(\kappa)$ have cone points consisting of one integral zero and integral residueless poles, then $\Omega^{tw} \cM_{g,n}(\kappa)$ admits a fully supported $\mathrm{SL}(2, \bR)$-invariant Lebesgue class measure on the complement of the locus of translation and homothety surfaces.

If all cone points are integral and poles are residueless, then the same holds after passing to a cover given by demanding that all arcs in a tree of arcs connecting the zeros are preserved in relative homology. 
\end{cor}

A natural question is whether $\cD(\kappa)$ also admits an invariant fully supported Lebesgue class measure. It is known to admit such a measure on the complement of the \emph{triangulable locus}, which is the locus of dilation surfaces that can be triangulated by saddle connections. We show the following.

\begin{thm}\label{T:area}
After passing to a cover as in Theorem \ref{T:main:general}, the triangulable locus of $\cD(\kappa)$ admits an $\mathrm{SL}(2, \bR)$ invariant Lebesgue class measure if and only if the triangulable locus in $\Omega^{tw} \cM_{g,n}(\kappa)$ admits an ``area function", i.e. a measurable $\mathrm{SL}(2, \bR)$-invariant map to $\bR_{>0}$ so that $\mathrm{diag}(\lambda, \lambda)$ changes ``area" by $\lambda^2$ for any $\lambda \in \bR^\times$.
\end{thm}

We will now describe the strategy for proving Theorem \ref{T:main:general} in the simplifying case of $\Omega^{tw} \cM_{g,1}(2g-2)$. Let $\pi_1 := \pi_1(\Sigma_g, p)$ where $p$ is the unique singularity. 

Recall that Masur-Veech measure on $\Omega \cM_g(2g-2)$ is constructed by passing to the universal cover of $\Omega \cM_g(2g-2)$, where the period map is a local diffeomorphism to $\mathrm{Hom}(\pi_1, \bC) \cong \bC^{2g}$, and pulling back Lebesgue measure, which is invariant under the action of the mapping class group. This measure then descends to $\Omega \cM_g(2g-2)$. 

Analogously, by Apisa \cite{Apisa-Period}, which extends work of Veech \cite{veech93}, the universal cover of $\Omega^{tw} \cM_{g,1}(2g-2)$ admits a framed mapping class group and $\mathrm{SL}(2, \bR)$-equivariant map that is a local diffeomorphism away from the locus of translation and homothety surfaces. This is the \emph{period map}, to $\mathrm{Hom}(\pi_1, \Dil)$. In analogy with the construction of Masur-Veech measure we would like to pull back a framed mapping class group invariant measure. We will characterize when such a measure exists.

Recall that the \emph{Johnson kernel} $\cK_{g,1} \le \Mod_{g,1}$ is the subgroup generated by Dehn twists around separating simple closed curves. Say that $\Gamma \leq \mathrm{Mod}_{g,1}$ is \emph{sufficiently large} if it contains the Johnson kernel and projects to a finite index subgroup of $\mathrm{Sp}(2g, \bZ)$. By Calderon-Salter \cite{CalderonSalter-RelHom}, framed mapping class groups are sufficiently large.

\begin{thm}\label{T:Main:InvariantForms}
If $g \geq 3$ and $\Gamma \leq \mathrm{Mod}_{g,1}$ is sufficiently large, then $\mathrm{Hom}(\pi_1, \Dil)$ admits a $\Gamma$-invariant Lebesgue class measure if and only if $\Gamma$ is contained in an $N$-framed mapping class group for some $N$.
\end{thm}

The framed mapping class group appears as (1) the stabilizer of an isotopy class of a vector field on a surface, (2) the fundamental group of strata of dilation surfaces by \cite{ABW1}, and (3) the image of the fundamental group of strata of translation surfaces in the mapping class group by Calderon-Salter \cite{CalderonSalter}. Theorem \ref{T:Main:InvariantForms} provides a new characterization of this subgroup and the closely related $N$-framed mapping class groups. It should be contrasted with work of Goldman \cite{goldman-symplectic}, which shows that character varieties defined by representations of surface groups into reductive Lie groups admit mapping class group invariant symplectic forms. 

Just as strata of holomorphic $k$-differentials inherit a volume form from Masur-Veech volume by passing to the holonomy cover, strata of twisted holomorphic $k$-differentials also inherit a volume form from strata of twisted holomorphic $1$-forms. The arguments in Apisa-Bainbridge-Wang \cite[Section 4]{ABW1} can be modified to show that strata of twisted holomorphic $k$-differentials are $K(\pi, 1)$ spaces, where $\pi$ is an $N$-framed mapping class group. The appearance of $N$-framed mapping class groups in Theorem \ref{T:Main:InvariantForms} is therefore expected. 

As we will see in Section \ref{S:Proof-OneConePoint}, the existence of an invariant measure on $\mathrm{Hom}(\pi_1, \Dil)$ is related to the existence of one on $\mathrm{Hom}(\pi_1, \Aff^+(\bR))$ where $\Aff^+(\bR)$ is the group of transformations $f(z) = az +b$ with $a,b$ real and $a > 0$. We will see that there is a map $\mathrm{Hom}(\pi, \Aff^+(\bR)) \ra H^1(\Sigma_g, \bR_{>0})$ whose fiber over a character $\chi$ can be identified with the space of twisted cocycles $Z^1(\Sigma_g, \bR_{\chi})$. The dimension of the fiber jumps over the trivial character. So let $\mathrm{Hom}^\circ(\pi_1, \Aff^+(\bR))$ be the result of deleting that fiber. This is a rank $2g-1$ vector bundle. The subbundle $B$ whose fiber over a character is the one-dimensional space of coboundaries is invariant under $\mathrm{Mod}_{g,1}$. Theorem \ref{T:Main:InvariantForms} will follow from the following. 

\begin{thm}\label{T:Main:InvariantSection}
Fix $g \geq 3$. Let $\Gamma \leq \mathrm{Mod}_{g,1}$ project to a finite index subgroup of $\mathrm{Sp}(2g, \bZ)$. Then $\bigwedge^{2g-1}\left( \mathrm{Hom}^\circ(\pi_1, \mathrm{Aff}^+(\bR)) \right)^*$ admits a $\Gamma$-equivariant nonvanishing section if and only if $\Gamma$ is contained in an $N$-framed mapping class group.  

This equivariant family of volume forms on fibers may be combined with the volume coming from the symplectic form on the base to form a $\Gamma$-invariant volume form on $\mathrm{Hom}^\circ(\pi_1, \mathrm{Aff}^+(\bR))$. It is ergodic when $\Gamma$ is sufficiently large. 
\end{thm}

The ergodicity statement relies on a similar statement in Ghazouani \cite{Ghazouani-MCGDynamicsAndAff}. Theorem \ref{T:Main:InvariantSection} provides another new characterization of framed mapping class groups and roughly says that, up to scaling, there is a unique choice of framed mapping class group invariant Lebesgue class measure on $\mathrm{Hom}(\pi_1, \Dil)$.

\subsection{Cohomology of the mapping class group}
Our method of proving Theorem \ref{T:main:general} is to formulate the obstruction to the existence of our desired measures in cohomological terms, and then show that these obstructions vanish on the covers we consider. The domain for our obstructions is the first cohomology of the mapping class group with certain twisted coefficient systems. 
The prototype of our results is the computation $H^1(\Mod_{g,1};H) \cong \bZ$ carried out by Morita in \cite{Morita-FamiliesOfJacobians}; here $H := H_1(\Sigma_g;\bZ)$. 

Since this group is nontrivial, vanishing of obstructions does not come for free, and one must moreover understand particular cocycles that represent generators. In the case of $H^1(\Mod_{g,1};H)$, this was considered by Trapp \cite{trapp}, and independently by Furuta (as recorded by Morita \cite[p. 569]{morita-secondary}). Both authors showed that $H^1(\Mod_{g,1};H)$ is generated by a ``change of winding number'' cocycle, implying that the pullback map $H^1(\Mod_{g,1}; H) \to H^1(\Mod_{g,1}[\xi]; H)$ is trivial, where $\Mod_{g,1}[\xi]$ is any framed mapping class group.

In the Appendix, we establish a certain generalization of these results which underpins the main results of our paper. The following is a paraphrase of what we show.

\noindent \textbf{Theorem \ref{T:MainMCGCohomology}}. (Paraphrase) {\em There is an isomorphism
\[
H^1(\Mod_{g,n+1};H_1(\Sigma_{g,n};\bZ)) \cong \bZ.
\]
and this group is generated by a ``change of relative winding number'' cocycle.
}

\subsection{Context} The moduli space of affine and projective structures with no cone points was first considered in Hubbard \cite{Hubbard-ProjectiveStructures} and Hejhal \cite{Hejhal}. Introducing cone points dates to the work of Gunning \cite{gunning78} and, via the correspondence between branched affine surfaces and regular connections developed in Novikov-Shapiro-Tahar  \cite{NovikovShapiroTahar} and Apisa-Bainbridge-Wright \cite{ABW2}, to Deligne \cite{deligne_equations_differentielles}. Veech \cite{veech93} placed coordinates on the moduli space of affine surfaces. Recently, strata of meromorphic $1$-forms have been intensely studied, see \cite{Boissy-components, Lee-components, ChenFaraco-Periods, moller_mullane_2023_teichmueller}. In the setting of twisted forms, the analogue involves letting $\kappa$ contain negative integers.

Fundamental work has appeared recently determining when dilation surfaces are triangulable by saddle connections (Duryev-Fougeron-Ghazouani \cite{DFG-Veech}), showing that they always possess cylinders (Boulanger-Ghazounai-Tahar \cite{AGT}), and assessing when a dense set of directions on a dilation surface are Morse-Smale (Tahar \cite{tahar}). In \cite{ABW1}, it is shown that every vector field on a Riemann surface that vanishes at finitely many points can be ``pulled tight to a dilation surface". This helps to explain the connection of dilation surfaces to recent work on flows on surfaces, see Ghazouani-Ulcigrai \cite{GhazouaniUlcigrai-Genus2}. See Abate-Tovena \cite{Abate-Tovena-VectorField} for another connection between affine surfaces and flows.




\noindent \textbf{Acknowledgements.} The first author was partially supported by NSF grant no. DMS-2304840. The second author was partially supported by NSF grant no. DMS-2338485.

\setcounter{tocdepth}{1} 
\tableofcontents

\section{Framed surfaces}\label{S:MCG}

As developed in the following sections, the problems under study in this paper are closely connected to the structure of certain ``framed mapping class groups'' - subgroups of the mapping class group of a surface that preserve the isotopy class of a framing. As detailed in Section \ref{S:DilationSurfaces}, dilation surfaces naturally carry a preferred framing, and their moduli is closely related to the framed mapping class group.

Here we recollect the necessary facts about framings on surfaces and the associated subgroups of mapping class groups that we will use in the sequel. For a more comprehensive treatment of framings and framed mapping class groups, see \cite{CalderonSalter}, and for more on framings in the context of dilation surfaces, see \cite{ABW1}.

\subsection{Perspectives on framings}\label{subsection:wnf} 
In this paper, we will need to understand framings on surfaces from several points of view.

\para{Perspective 1: topology} A {\em framing} $\xi$ of a surface $\Sigma_{g,n}$ is a trivialization of the tangent bundle. This can be viewed as a pair $\xi_1, \xi_2$ of vector fields that are everywhere linearly independent, or more abstractly as a section of the principal bundle associated to the tangent bundle. Throughout, we will work with framings up to {\em isotopy}, i.e. up to homotopy of sections of the associated principal bundle. Accordingly we will be somewhat lax in our terminology, speaking of ``framings'' when what is meant precisely is ``isotopy class of framing''.  Since $\Sigma_{g,n}$ is oriented, a framing can be specified up to isotopy by a choice of a single non-vanishing vector field $\xi_1$ (extending this to a framing by choosing a Riemannian metric and taking $\xi_2$ to be orthogonal and positively-oriented).

\para{Perspective 2: winding number functions} A framing of $\Sigma_{g,n}$ gives rise to a {\em winding number function}. Let $\cS(\Sigma_{g,n})$ denote the set of isotopy classes of oriented simple closed curves on $\Sigma_{g,n}$. A winding number function is a map $W: \cS(\Sigma_{g,n}) \to \bZ$ satisfying the ``twist-linearity'' and ``homological coherence'' properties; see the paragraph ``Properties of winding number functions'' below. A non-vanishing vector field $\xi$ gives rise to a winding number function $W_\xi$ by the following procedure: given an oriented simple closed curve $c$, measure the total winding number of the forward-pointing tangent vector of $c$ relative to $\xi$. This is well-defined on the level of isotopy and in fact {\em encodes} the isotopy class of the framing; again see \cite[Section 2]{CalderonSalter}.

\para{Perspective 3: cohomology} Winding number functions are not cohomology classes {\em on $\Sigma_{g,n}$}, in the sense that $W: \cS(\Sigma_{g,n}) \to \bZ$ does not descend to a homomorphism $\overline W: H_1(\Sigma_{g,n}) \to \bZ$. On the other hand, winding number functions can be interpreted as cohomology classes on the unit tangent bundle $UT\Sigma_{g,n}$, as we explain. Let $\xi$ be a non-vanishing vector field on $\Sigma_{g,n}$. Normalizing length with respect to some Riemannian metric, $\xi$ can be viewed as a section map $\xi: \Sigma_{g,n} \to UT\Sigma_{g,n}$. The Poincar\'e dual to the image of $\xi$ determines an element\footnote{Since $UT\Sigma_{g,n}$ is noncompact, some care needs to be taken to make this precise, but we will elide this inessential technicality here.} of $H^1(UT\Sigma_{g,n})$ which we write (temporarily) as $PD(\xi)$. Such $PD(\xi)$ determines a winding number function via the {\em Johnson lift}: given an oriented simple closed curve $c \subset \Sigma_{g,n}$, let $\hat c \subset UT\Sigma_{g,n}$ denote the curve sitting above $c \subset \Sigma_{g,n}$ endowed with its forward-pointing unit tangent vector. Then the winding number of $c$ about $\xi$ can be computed homologically:
\[
W_\xi(c) = PD(\xi)(\hat c).
\]
Accordingly, we will conflate perspectives 2 and 3, thinking of a winding number function as an element of $H^1(UT\Sigma_{g,n};\bZ)$ and conversely. To unify notation, we will write $W_\xi$ to refer to both the winding number function and the associated cohomology class. 

It is important to understand the set of winding number functions as a subset of $H^1(UT\Sigma_{g,n};\bZ)$. This has the structure of an affine space.

\begin{prop}\label{prop:framingchar}
    An element $\phi \in H^1(UT\Sigma_{g,n}; \bZ)$ arises as $W_\xi$ for some framing $\xi$ if and only if $\phi(\zeta) = 1$, where $\zeta \in H_1(UT\Sigma_{g,n}; \bZ)$ is the loop around the $S^1$-fiber of $UT\Sigma_{g,n} \to \Sigma_{g,n}$, endowed with the positive orientation. In particular, if $W_\xi, W_\eta$ are two such classes, then 
    \[
    W_\xi = W_\eta + p^*v
    \]
    for some $v \in H^1(\Sigma_{g,n};\bZ)$, where $p: UT\Sigma_{g,n} \to \Sigma_{g,n}$ is the projection.
\end{prop}
\begin{proof}
    See \cite[Section 2]{CalderonSalter}.
\end{proof}

\para{Properties of winding number functions}
Here we discuss the two characterizing properties of winding number functions: twist-linearity and homological coherence.

\para{Twist-linearity}
The twist-linearity property asserts that if $c, d$ are simple closed curves, then for any winding number function $W$,  
\begin{equation}\label{eqn:twistlin}
    W(T_d(c)) = W(c) + \pair{c,d}W(d),
\end{equation}
where $\pair{c,d}$ denotes the algebraic intersection pairing. In particular, if $d$ is a {\em separating curve}, then $W(T_d(c)) = W(c)$ for all simple closed curves $c$. This property is easy to establish from either the winding number or cohomological perspective.

\para{Homological coherence} The second characterizing property of a winding number function is {\em homological coherence}. Suppose that $S \subset \Sigma_{g,n}$ is a subsurface with boundary components $c_1, \dots, c_k$. Orient each $c_i$ so that $S$ lies to the left. Then homological coherence asserts that for any winding number function $W$,
\begin{equation}\label{eqn:homcoh}
    \sum W(c_i) = \chi(S),
\end{equation}
where $\chi(S)$ denotes the Euler characteristic. This is a consequence of the Poincar\'e-Hopf theorem, or (equivalently) of the equality in $H_1(UT\Sigma_{g,n};\bZ)$
\[
\sum [\hat c_i] = \chi(S)[\zeta].
\]

\subsection{The framed mapping class group} The set of isotopy classes of framings on $\Sigma_{g,n}$ carries an action of the mapping class group $\Mod_{g,n}$, by pullback. For a framing $\xi$ (equivalently, a winding number function, or a suitable cohomology class on $UT\Sigma_{g,n}$), the stabilizer of $\xi$ under this action is written $\Mod_{g,n}[\xi]$ and is called a {\em framed mapping class group}. A basic property of framed mapping class groups that will be occasionally invoked is that, by the discussion of the previous paragraph, every Dehn twist $T_d$ about a separating curve is contained in every $\Mod_{g,n}[\xi]$. In the Introduction, we defined $\cK_{g,1} \le \Mod_{g,1}$ as the subgroup generated by separating twists, implying the following statement.

\begin{prop}\label{prop:kginframed}
    For any framing $\xi$ of $\Sigma_{g,1}$, there is a containment
    \[
    \cK_{g,1} \le \Mod_{g,1}[\xi]. 
    \]
\end{prop}

\subsection{Winding numbers of arcs} \label{SS:relwnf}
While our ultimate interest is in surfaces with punctures/marked points, it will be necessary in places to consider surfaces with boundary components as well. Here there is a theory of {\em relative winding number functions} that assign winding numbers to arcs as well as curves. We adopt a slightly different perspective here from the original reference \cite[Section 2]{CalderonSalter}, making it clear that relative winding number functions can be formulated as {\em relative} cohomology classes on the unit tangent bundle, exactly parallel to the discussion above.

Let $\Sigma_{g}^b$ denote a surface of genus $g$ with $b$ boundary components. As the mapping class group $\Mod(\Sigma_{g}^b)$ fixes the boundary pointwise, there is a refined notion of {\em relative isotopy} of framings: framings $\phi, \psi$ are {\em relatively isotopic} if they are isotopic through isotopies restricting to the identity on the boundary. 

As in the ``absolute'' setting above, one can view a non-vanishing vector field $\xi$ as a section $\xi: \Sigma_{g}^b \to UT\Sigma_{g}^b$, determining a class $[\xi] \in H_2(UT\Sigma_{g}^b,\wt \partial)$, where $\wt \partial \subset UT \Sigma_{g}^b$ denotes the boundary, a union of $b$ $2$-tori. In order to use this data to determine a relative winding number function, one must specify slightly more information. Let $\wt P$ be a set of $b$ points, one on each boundary component of $UT\Sigma_{g}^b$. A relative winding number function for arcs based only at $\wt P$ can be described as an element of $H^1(UT\Sigma_{g}^b, \wt P)$. The next lemma shows that this can indeed be induced from the data of a non-vanishing vector field.

\begin{lem}\label{L:relwnf}
                There is a well-defined intersection pairing
                \[
                H_1(UT\Sigma_g^b, \wt P) \times H_2(UT\Sigma_g^b, \wt \partial \setminus \wt P) \to \bZ
                \]
                realizing an isomorphism
                \[
                H^1(UT\Sigma_g^b, \wt P) \cong H_2(UT\Sigma_g^b, \wt \partial \setminus \wt P).
                \]
                Consequently, a non-vanishing vector field $\xi: \Sigma_{g}^b \to UT\Sigma_{g}^b$ that avoids $\wt P$ determines a relative fundamental class in $H_2(UT\Sigma_g^b, \wt \partial \setminus \wt P)$ and hence a relative winding number function $W_\xi^+ \in H^1(UT\Sigma_{g}^b, \wt P)$.
            \end{lem}
            \begin{proof}
                The general form of Lefschetz duality asserts that if $M$ is a compact $n$-manifold with boundary $\partial M$, and if $\partial M$ decomposes as a union $\partial M = X \cup Y$, where $X,Y$ are compact $(n-1)$-manifolds with common boundary $\partial X = \partial Y$, then there is an isomorphism
                \[
                H^i(M,X) \cong H_{n-i}(M, Y),
                \]
                realized by the intersection pairing. Enlarging $\wt P \subset \wt \partial = \partial UT\Sigma_g^b$ to a regular neighborhood, this gives the above claim.
            \end{proof}

In general, the winding number of an arc with one or more endpoints at a puncture/marked point is ill-defined. This can be seen as follows: let $\alpha$ be an arc based at a marked point $p \in \Sigma_{g,n}^b$, and let $\gamma$ be a curve encircling $p$. If $W$ were well-defined, it would have to satisfy the twist-linearity formula, and so  $W(T_\gamma(\alpha)) = W(\alpha) + W(\gamma)$. On the other hand, $T_\gamma(\alpha)$ is isotopic to $\alpha$. We come to the following conclusion.
\begin{prop}\label{prop:ewnf}
    Let $\Sigma_{g,n}$ be a surface of genus $g$ with $n$ marked points. Let $\xi$ be a framing of $\Sigma_{g,n}$ with associated winding number function $W_\xi$. Let $P_0$ denote the set of marked points with zero winding number (i.e. those marked points $p$ for which $W_\xi(\gamma) = 0$ for $\gamma$ a small curve encircling $p$). Let $\cS^+_0(\Sigma_{g,n})$ denote the set of isotopy classes of oriented simple closed curves and simple arcs with both endpoints in $P_0$. Then $W_\xi$ admits an extension $W_\xi^+: \cS^+_0(\Sigma_{g,n})\to \bZ$.
\end{prop}
\begin{proof}
    To construct $W_\xi^+$, blow up each marked point to a boundary component and extend $\xi$ to a framing on the blowup $\Sigma_g^n$. Choose basepoints in $UT\Sigma_g^n$ at each of the boundary components corresponding to $P_0$. For an arc $\alpha$ with endpoints in $P_0$, lift $\alpha$ to an arc $\tilde \alpha$ on $\Sigma_g^n$ based at the chosen basepoints. The isotopy class of $\tilde \alpha$ is unique only up to twists about each of the boundary components. Since the winding number of each boundary component is zero by assumption, it follows by twist-linearity that the winding number of all such $\tilde \alpha$ are equal, and hence that $W_\xi^+$ is well-defined. 
\end{proof}

\section{The affine holonomy bundle: Theorems \ref{T:Main:InvariantForms} and \ref{T:Main:InvariantSection}}\label{S:Proof-OneConePoint}

\subsection{The affine holonomy bundle: basics}\label{SS:affholbasics}
Throughout, let $K$ be either $\bR$ or $\bC$. The {\em affine group over $K$} is the group
\[
\Aff(K) = \{f(z) = az + b \mid a \in K^\times, b \in K\}.
\]
For $K = \bR$ we also define the subgroup
\[
\Aff^+(\bR) = \{f(z) = az + b \mid a \in \bR_{>0}, b \in \bR\}.
\]
The group structure on $\Aff(K)$ is by composition: if $f = az + b$ and $g = cz + d$, then the product $f \circ g$ is given by
\begin{equation}\label{E:affmult}
    f \circ g = (ac)z + (b+ad).
\end{equation}

The representation variety $\Hom(\pi_1\Sigma_g, \Aff(K))$ admits a projection
\[
\Pi: \Hom(\pi_1\Sigma_g, \Aff(K)) \to H^1(\Sigma_g; K^\times)
\]
that sends $f: \pi_1\Sigma_g \to \Aff(K)$ given by $f(\gamma) = \chi(\gamma) z + \lambda(\gamma)$ to the character $\chi \in H^1(\Sigma_g;K^\times)$. Conceptually, this is just the {\em derivative} of the affine map.

To describe the fibers of this map, we recall the notion of a {\em $1$-cocycle}. We will encounter this notion in two places: here, in the context of the fibers of the derivative map, as well as in the Appendix, where we will consider cocycles for the mapping class group. Accordingly, we formulate this discussion in general terms. We give a minimal overview of what we will need; for more details, see the standard reference \cite{brown}. 

Let $X$ be a space with fundamental group $G$, and let $V$ be a $\bZ[G]$-module, i.e. an abelian group equipped with an action of $G$. A {\em $1$-cocycle} (also known as a {\em crossed homomorphism}) is a map $f: G \to V$ that satisfies the cocycle equation
\begin{equation}\label{eqn:cocycle}
    f(xy) = f(x) + x \cdot f(y)
\end{equation}
for all $x,y \in G$. The space of twisted cocycles is written as $Z^1(X,V)$, or else as $Z^1(G,V)$, depending on whether one wants to emphasize the space or its fundamental group. There is a map $\delta: V \to Z^1(G,V)$ sending $v \in V$ to the cocycle 
\[
\delta(v)(x) = (1-x)\cdot v.
\]
Such $\delta$ is called a {\em coboundary (in the sense of group cohomology)}. The quotient
\[
H^1(G,V) = Z^1(G,V) / \Im(\delta)
\]
is the first cohomology of $G$ (or $X$)\footnote{Experts will notice that we are conflating the cohomology of the space with that of its fundamental group. This is harmless since these coincide in degree one, but the reader should be aware that this is not necessarily true in higher degrees.} with twisted coefficient module $V$.

In our immediate context, we consider $G = \pi_1\Sigma_g$ and the $G$-module $\bC_\chi$ with underlying group $\bC$, carrying an action of $\pi_1\Sigma_g$ via a character $\chi: \pi_1\Sigma_g \to \bC^\times$. For brevity, we write
\[
Z^1_\chi := Z^1(\pi_1\Sigma_g, \bC_\chi).
\]

\begin{lem}
    Fix $\chi \in H^1(\Sigma_g; K^\times)$. The fiber $\Pi^{-1}(\chi) \subset \Hom(\pi_1 \Sigma_g, \Aff(K))$ is in bijection with the space of cocycles $Z^1_\chi$, via the correspondence
    \[
    \lambda \in Z_\chi^1 \longleftrightarrow f(\gamma) = \chi(\gamma)z + \lambda(\gamma).
    \]
\end{lem}
\begin{proof}
    Suppose $f \in \Hom(\pi_1 \Sigma_g, \Aff(K))$ is of the form $f(\gamma) = \chi(\gamma)z + \lambda(\gamma)$. As $f: \pi_1\Sigma_g \to \Aff(K)$ is a homomorphism, for any $\gamma_1, \gamma_2 \in \pi_1\Sigma_g$, there must be an equality $f(\gamma_1 \gamma_2) = f(\gamma_1) \circ f(\gamma_2)$, so that by \eqref{E:affmult},
    \[
    \chi(\gamma_1\gamma_2)z + \lambda(\gamma_1 \gamma_2) = (\chi(\gamma_1) \chi(\gamma_2)) z + (\lambda(\gamma_1) + \chi(\gamma_1) \lambda(\gamma_2)).
    \]
    Thus $f(\gamma) = \chi(\gamma)z + \lambda(\gamma)$ determines a homomorphism if and only if $\lambda \in Z_\chi^1$.
\end{proof}

Let $\Sigma_g$ be a genus $g$ surface with a basepoint $p$. Fix a set of simple closed curves $S = \{a_1, b_1 \hdots, a_g, b_g \} \subseteq \pi_1(\Sigma_g, p)$ so that 
\[
\pi_1 := \pi_1(\Sigma_g, p) = \langle a_1, b_1, \dots, a_{g}, b_g | [a_1, b_1] \hdots [a_g,b_g] = 1 \rangle.
\]

\begin{lem}\label{L:Zdim}
    If $\chi = \mathbf{1}\in H^1(\Sigma_g;K^\times)$ is the trivial character, then $Z_\chi^1 \cong K^{2g}$, and otherwise $Z_\chi^1 \cong K^{2g-1}$.
\end{lem}
\begin{proof}
    Any $\lambda \in Z_\chi^1$ is determined by its values $(\lambda(s))_{s \in S}$, subject only to the condition
    \begin{equation}\label{eqn:twistedcocyclecondition}
0 = \sum_{i=1}^g \lambda\left( [a_i, b_i] \right) = \sum_{i=1}^g \left( (1-\chi(b_i)) \lambda(a_i) - (1-\chi(a_i))\lambda(b_i)  \right). 
\end{equation}
This realizes $Z_\chi^1$ as the solution set of this equation in $K^{2g}$. If $\chi = \mathbf{1}$ then this equation is trivial, otherwise its solutions form a subspace of codimension one.
\end{proof}

We define
\[
H^1(\Sigma_g;K^\times)^\circ := H^1(\Sigma_g;K^\times) \setminus \{\mathbf{1}\}
\]
and
\[
\cV_K := \Pi^{-1}(H^1(\Sigma_g;K^\times)^\circ) \le \Hom(\pi_1,\Aff(K)).
\]
In the case $K = \bR$, we also define $\cV_\bR^+$ as the corresponding object replacing $\Aff(\bR)$ with $\Aff^+(\bR)$.

Following Lemma \ref{L:Zdim}, we expect $\Pi:\cV_K \to H^1(\Sigma_g;K^\times)^\circ$ to be a vector bundle of rank $2g-1$. This is indeed the case, and we will exhibit this explicitly. To that end, we will say that $a_i \in S$ is the {\em dual} of $b_i$, and conversely.

\begin{lem}\label{L:StandardBasisOfCocycles}
Fix $c \in S$ with dual $d$, and fix $\chi \in H^1(\Sigma_g;K^\times)^\circ$. The map $p_c: Z^1_\chi \ra K^{2g-1}$ given by sending $\lambda$ to $(\lambda(s))_{s \in S - \{c\}}$ is a bijection if and only if $\chi(d) \ne 1$. 
\end{lem}
\begin{proof}
For notational simplicity, we will consider only the case $c = a_i$ and $d = b_i$. Returning to the cocycle equation \eqref{eqn:twistedcocyclecondition}, if $(1-\chi(b_i)) = 0$, then, since $\chi \not \equiv 1$, \eqref{eqn:twistedcocyclecondition} determines a nontrivial hyperplane that contains the image of $p_{a_i}$. Conversely, if $(1-\chi(b_i)) \ne 0$, then given any element $v \in K^{2g-1}$ there is a unique choice of $\lambda(a_{i})$ that completes $v$ into a tuple $(\lambda(a_1), \dots, \lambda(b_g))$ satisfying \eqref{eqn:twistedcocyclecondition}.
\end{proof}

Define $\alpha_i: Z_\chi^1 \ra K$ by $\alpha_i(\lambda) := \lambda(a_i)$ and $\beta_i: Z_\chi^1 \ra K$ by $\beta_i(\lambda) := \lambda(b_i)$ for $1 \leq i \leq g$. Rewriting Equation \ref{eqn:twistedcocyclecondition} gives
\begin{equation}\label{eqn:twistedcocyclecondition2}
0 = \sum_{i=1}^g (1-\chi(b_i)) \alpha_i - (1 - \chi(a_i))\beta_i.
\end{equation}

Given a character $\chi: \pi_1 \ra K^\times$, set $x_i := \chi(a_i)$ and $y_i = \chi(b_i)$. 


\begin{cor}\label{C:FiberBundle}
There is a finite collection of Zariski open sets covering $H^1(\Sigma_g, K^\times)^\circ$ over which $\cV_K$ is trivial. The transition functions between different trivializations are rational. 
\end{cor}
\begin{proof}
For each $c \in S$, let $U_c \subseteq H^1(\Sigma_g, K^\times)^\circ$ be the subset where $\chi(d) \ne 1$, where $d$ is the dual of $c$. By Lemma \ref{L:StandardBasisOfCocycles}, we may use $(\alpha_i, \beta_i)_{i=1}^g$ (ignoring the function corresponding to $c$) to trivialize the bundle over $U_c$. By Equation \ref{eqn:twistedcocyclecondition2}, the change of trivialization matrix has entries that are rational functions of $(x_i, y_i)_{i=1}^g$.
%
\end{proof}

\subsection{Fiberwise volume forms}

Our goal is to construct a fiberwise volume form on $\cV_\bR$ that is equivariant with respect to the action of some subgroup of the mapping class group. We are therefore naturally led to a study of the (dual) determinant line bundle. Although our ultimate interest is in the setting $K = \bR$, our method of proof will require an appeal to the Nullstellensatz; for that reason, we continue to allow the possibility of (and occasionally necessitate) $K = \bC$.

\begin{lem}\label{L:TrivialDeterminantBundle}
$\bigwedge^{2g-1} \cV_K^*$ is a trivial line bundle.
\end{lem}
\begin{proof}
Let $c \in S$ have dual $d$, and write $\gamma: Z_\chi^1 \to K$ for the associated functional $\gamma(\lambda) = \lambda(c)$; define $\delta$ similarly for $d$. Define the $2g-1$-form on $U_c$
\[
\omega_c := (1-\chi(d))^{-1} \alpha_1 \wedge \beta_1 \wedge \dots \wedge  \widehat{\gamma} \wedge \dots \wedge \alpha_g \wedge \beta_g.
\]

Let $s \in S$ be some other curve, with dual $t$; write $\sigma, \tau$ for the associated functionals. Define $\epsilon_s$ to be $1$ if $s = \alpha_i$ for some $i$ and to be $-1$ if $s = \beta_i$. By Equation \ref{eqn:twistedcocyclecondition2}, on $U_{c} \cap U_{s}$ we see that 
\[
\sigma = -\epsilon_s(1-\chi(t))^{-1}\epsilon_c(1 - \chi(d))\gamma + \text{(terms not involving $\sigma$ or $\gamma$)}).
\]

Plugging in this expression for $\sigma$ into the equation for $\omega_c$, we see that $\omega_c = \omega_s$ on $U_c \cap U_s$, and so assembles into a non-vanishing section of the determinant line bundle.
\end{proof}

Let $\omega$ be the non-vanishing section of $\bigwedge^{2g-1}\cV_K^*$ defined in Lemma \ref{L:TrivialDeterminantBundle}. Abusing notation, let $\omega_\chi$ be the section defined over the character $\chi$. This is a (possibly complex) multiple of the volume form on $Z^1_\chi$ and hence it defines a measure, which we write as $|\omega_\chi|$. We recall the following standard fact from linear algebra. 

\begin{lem}\label{L:LinearAlgebra}
Fix a positive integer $N$. Suppose that $(e_1, \hdots, e_N)$ and $(f_1, \hdots, f_N)$ are bases of vector spaces $V$ and $W$. Let $(x_1, \hdots, x_N)$ and $(y_1, \hdots, y_N)$ be the dual bases of $V^*$ and $W^*$, i.e. bases so that $x_i(e_j) = y_i(f_j) = \delta_{ij}$ where $\delta_{ij}$ is $1$ if $i = j$ and $0$ otherwise. Set $\omega_V := x_1 \wedge \hdots \wedge x_N$ and $\omega_W = y_1 \wedge \hdots \wedge y_N$. Let $T: V \ra W$ be a linear isomorphism, which specifies a matrix using the indicated bases. Then $T^* \omega_W = \det(T) \omega_V$. If $V$ and $W$ are real, then $|\det(T)| T_*|\omega_V| = |\omega_W|$.
\end{lem}

For $f \in \Mod_{g,1}$, there is an action $f^*: Z^1_\chi \ra Z^1_{f^* \chi}$ given by sending $\lambda \in Z^1_\chi$ to $f^* \lambda := \lambda \circ f$.  This induces a map $f_*: \bigwedge^{2g-1} (Z^1_\chi)^* \ra \bigwedge^{2g-1} (Z^1_{(f^{-1})^*\chi})^*$. For simplicity, define $f_* \chi := (f^{-1})^*\chi$. Define the cocycle $A(f, \chi)$ to be the function 
\[
A(f, \cdot): H^1(\Sigma_g, K^\times)^\circ \ra K^\times
\]
so that
\[ f_* \omega_\chi = A(f, \chi) \omega_{f_* \chi}\]
This is a \emph{cocycle} in the sense that 
\[
A(fg, \chi) = A(f, g_*\chi)A(g, \chi).
\]

The restriction of the cocycle to a subgroup $G \leq \mathrm{Mod}_{g,1}$ is a \emph{coboundary (in the sense of dynamics)} if there is a function $F: H^1(\Sigma_g, K^\times)^\circ \ra K^\times$ so that $A(f, \chi) = \frac{F(f_*\chi)}{F(\chi)}$ for $f \in G$. When this occurs, $F(\chi)\omega_\chi$ is a nowhere vanishing $G$-equivariant section of $\bigwedge^{2g-1} \cV_K^*$.

Set $(c_1, \hdots, c_{2g}) = (a_1, b_1, \hdots, a_g, b_g)$. On $U_{c_{2g}} = U_{b_g}$, we have a basis of $Z^1_\chi$ given by cocycles $\lambda_i$ defined by the condition that $\lambda_i(c_j) = \delta_{ij}$ for $1 \leq i, j \leq 2g-1$. Define $\gamma_i: Z^1_\chi \ra K$ by $\gamma_i(\lambda) = \lambda(c_i)$. Recall that $H^1(\Sigma_g, K^\times) = (K^\times)^{2g} \subseteq K^{2g}$; in this way the notion of a rational function is defined on $H^1(\Sigma_g, K^\times)$ (or a subspace).

\begin{lem}\label{L:CocycleIsRational}
$A(f, \cdot): H^1(\Sigma_g, K^\times)^\circ \ra K^\times$ is a rational function. For $\chi \in H^1(\Sigma_g, K^\times)^\circ$ so that $\chi$ and $f_* \chi \in U_{b_{g}}$, we have (recalling $x_i := \chi(a_i)$ and $y_i := \chi(b_i)$),
\[ 
A(f, \chi) = \frac{f_*\chi(a_g)-1}{\chi(a_g)-1} \det\left( \lambda_i(f(c_j)) \right)_{i,j = 1}^{2g-1} \in \bQ(x_1, y_1, \hdots, x_{g}, y_{g}).  
\]
\end{lem}
\begin{proof}
The first claim follows from the second, which we now show. Consider the following commutative diagram,
\begin{center}
\begin{tikzcd}
Z^1_{f_*\chi} \arrow[r, "f^*"] \arrow[d, "\lambda \ra (\lambda(c_i))"]
& Z^1_{\chi} \arrow[d, "\lambda \ra (\lambda(c_i))"] \\
K^{2g-1} \arrow{r}[swap]{(\lambda_i(f(c_j)))_{ij}}
& K^{2g-1}
\end{tikzcd}
\end{center}
The basis $(\lambda_i)_{i=1}^{2g-1}$ of $Z^1$ is dual to the basis $(\gamma_i)_{i=1}^{2g-1}$ of $(Z^1)^*$ where we have suppressed the dependence on $\chi \in H^1(\Sigma_g, K^\times)^\circ$. By definition, $\omega_\chi = (1-\chi(a_g))^{-1} \gamma_1 \wedge \hdots \wedge \gamma_{2g-1}$. Set $T = f^*$. Written as a matrix with respect to the bases $(\lambda_i)_{i=1}^{2g-1}$ this is $(\lambda_i(f(c_j)))_{ij}$. By Lemma \ref{L:LinearAlgebra}, 
\[ (1-\chi(a_g)) A(f, \chi) \omega_{f_* \chi} = T^*\left( 1-\chi(a_g) \right) \omega_\chi = (\det T) (1-f_*\chi(a_g)) \omega_{f_* \chi}, \]
from which the result follows.
\end{proof}


Let $R$ be an abelian group, written multiplicatively, and let $C: \Mod_{g,1} \times H^1(\Sigma_g; R) \ra R$ be a cocycle, i.e. a map satisfying $C(fg, \chi) = C(f, g_* \chi)C(g, \chi)$.
The cocycle is \emph{multiplicative (in the second factor)} if $C(f, \chi_1\chi_2) = C(f,\chi_1)C(f, \chi_2)$. For such cocycles, $C(f, \cdot)$ is an element of 
\[
\Hom(H^1(\Sigma_g, R), R) \cong H_1(\Sigma_g, R).
\] 
This determines a map $C: \Mod_{g,1} \ra H_1(\Sigma_g, R)$ that satisfies
\[ C(fg, \cdot) = C(f, g_*\cdot)C(g, \cdot) = g^*\left( C(f, \cdot) \right) C(g, \cdot)  \]
 where $g^* := (g^{-1})_*$. Define $D(f, \chi) := C(f^{-1}, \chi)$. Then 
\[
D(fg, \cdot) = C(g^{-1}f^{-1}, \cdot) = (f^{-1})^*\left( C(g^{-1}, \cdot) \right) C(f^{-1}, \cdot) = f_*(D(g, \cdot)) D(f, \cdot)
\]
This is a crossed homomorphism, i.e. an element of $Z^1(\Mod_{g,1}, H_1(\Sigma_g, R))$. If $D$ is a coboundary in the sense of group cohomology, then there is a homology class $v$ so that $D(f, \cdot) = (f_*v)(v)^{-1}$ and so 
\[ C(f, \chi) = D(f^{-1}, \chi) = \frac{\chi(f^* v)}{\chi(v)} = \frac{(f_* \chi)(v)}{\chi(v)} \]
and so $C$ is a coboundary in the sense of dynamics.

We now specialize to the case $K = \bC$.

\begin{lem}\label{L:NullstellensatzArgument}
Fix $d \in \bZ_{>0}$. If $f: \bC^d \ra \bC \cup \{ \infty \}$ is a rational function whose only zeros and poles occur on $\bC^d - (\bC^\times)^d$, then $f$ is a monomial. 
\end{lem}

\begin{proof}
Write $f = \frac{g}{h}$ where $g$ and $h$ are polynomials with no irreducible factors in common. By hypothesis, $g$ and $h$ only vanish when some $x_i$ vanishes. Let $V(g)$ be the vanishing set of $g$. Since $V(g)$ is contained in the union of coordinate planes, the Nullstellensatz implies that the radical of $(g)$ contains $(x_1 \cdot \hdots \cdot x_d)$, which is the ideal of functions in $\bC[x_1, \hdots, x_d]$ vanishing on the coordinate axes. The radical of $g$ is the principal ideal $(G)$, which is generated by the product $G$ of distinct monic irreducible polynomials that divide $g$. Since $G \mid x_1 \hdots x_d$, it follows that the only monic irreducible polynomials dividing $g$ belong to the set $\{x_1, \hdots, x_d\}$. The same argument show that the same holds for $h$ and hence the result follows. 
\end{proof}

\begin{cor}\label{C:PrelimCocycleResult}
For $g\ge 2$, $A(f, \chi)$ is a multiplicative cocycle:
\[
A(\cdot, \cdot) \in Z^1(\Mod_{g,1},H^1(\Sigma_{g,1}; \bC^\times)^*).
\]
\end{cor}
\begin{proof}
By Lemma \ref{L:CocycleIsRational} $A(f, \chi)$ is a rational function defined over $\bQ$. So we view it as a function $A(f, \cdot): \bC^{2g-1} \ra \bC \cup \{\infty\}$. Its only zeros and poles occur on $\left( \bC^{2g-1} - (\bC^\times)^{2g-1} \right) \cup \{ (1, \hdots, 1) \}$. By the Weierstrass preparation theorem, multivariable rational functions have no isolated zeros or poles, so $(1, \hdots, 1)$ is not a zero or pole. So by Lemma \ref{L:NullstellensatzArgument}, $A(f, \chi) = c_f \prod_{i=1}^{g} x_i^{n_{i,f}} y_i^{m_{i,f}}$ where $c_f \in \bC^\times$ and $n_{i,f}$ and $m_{i,f}$ are integers. The cocycle condition implies that the map $\Mod_{g,1} \ra \bC^\times$ sending $f$ to $c_f$ is a homomorphism, e.g. by evaluating it at $\chi = \mathbf{1}$. Since $A$ is defined over $\bQ$, it follows that $c_f \in \bQ^\times$. The torsion subgroup of $\bQ^\times$ is $\{\pm 1\}$. Since the abelianization of $\mathrm{Mod}_{g,1}$ is trivial for $g \geq 3$ and torsion for $g \geq 2$, $c_f = 1$ for $g \geq 3$ and $c_f = \pm 1$ for $g = 2$. 

 When $g = 2$, we note that if $f$ is the Dehn twist around $a_2$, then $A(f, \chi) = 1$ by Lemma \ref{L:CocycleIsRational}. It follows that $c_f = 1$. Since $f$ projects to a generator of the abelianization of $\mathrm{Mod}_{g, 1}$, the argument is complete.
 \end{proof}

Recall that the kernel of the forgetful map $\Mod(\Sigma_{g,1}) \to \Mod(\Sigma_g)$ is isomorphic to $\pi_1\Sigma_g$ (based at the distinguished point of $\Sigma_{g,1}$, viewed as a marked point as opposed to a puncture). An element $\gamma \in \pi_1\Sigma_g$ is realized as a mapping class on $\Sigma_{g,1}$ by ``pushing'' the marked point along the path specified by $\gamma$; such mapping classes are called {\em point push maps}.

\begin{lem}\label{L:PointPush}
When $f$ is a point push around a closed curve $\gamma$, $A(f, \chi) = \chi(\gamma)^{2g-2}$.
\end{lem}
\begin{proof}
As the action of $f$ on $\pi_1(\Sigma_g,p)$ is by conjugation, it acts trivially on homology, and so $f^*\chi = \chi$ for any character $\chi$. Therefore, for each character $\chi$, the restriction of $A(\cdot, \chi)$ to the point pushing subgroup determines a homomorphism from $\pi_1(\Sigma_g, p)$ to $\bC^\times$. Thus it suffices to prove the claim for point pushes around basic curves $c_k$ (recalling our convention that $(c_1, \dots, c_{2g}) = (a_1,b_1, \dots, a_g,b_g)$). For $1 \leq i, j \leq 2g $,
\begin{align*}
    A_{ij} := \lambda_i(c_kc_jc_k^{-1}) =& \delta_{ik} + \chi(c_k) \delta_{ ij } - \chi(c_j) \delta_{ik}\\ 
    =& (1-\chi(c_j)) \delta_{ik} + \chi(c_k) \delta_{ ij }. 
\end{align*}

The matrix $(A_{ij})$ thus has nonzero entries only in row $k$ and down the diagonal, so its determinant is the product of its diagonal entries. Note that $A_{kk} = 1$ and that $A_{ii} = \chi(c_k)$ for $i\ne k$. The claim now follows from Lemma \ref{L:CocycleIsRational}.  
\end{proof}

Corollary \ref{C:PrelimCocycleResult} establishes that the restriction of $A$ to a map $\mathrm{Mod}_{g,1} \times H^1(\Sigma_{g,1}, \bR_{>0}) \ra \bR_{>0}$ determines a cocycle in $Z^1(\Mod_{g,1}, H^1(\Sigma_{g,1}, \bR_{>0})^*)$ when $g \geq 3$. In the Appendix, we take up the question of describing this space of cocycles and the associated cohomology group. In the present setting of a single puncture, the necessary analysis was carried out by Morita \cite{Morita-FamiliesOfJacobians}, who computed $H^1(\Mod_{g,1}; H^1(\Sigma_g))$, and Furuta and Trapp (\cite{morita-secondary,trapp}), who showed that the space is spanned by a {\em change of winding number functional}
\[
f \in \Mod_{g,1} \mapsto (f^*-1)W_\xi \in H^1(\Sigma_{g,1};\bZ),
\]
where $W_\xi \in H^1(UT\Sigma_{g,1};\bZ)$ is the class associated to a framing $\xi$ of $\Sigma_{g,1}$ as in Section \ref{subsection:wnf}. This leads to the following description of our cocycle $A$.

\begin{lem}\label{L:ItsACoboundary}
There is a framing $\xi$ so that 
\[
A(f, \chi) = \chi\left(PD((f^*-1)W_\xi) \right).
\]
Here, $(f^*-1)W_\xi$ lies in the subspace $H^1(\Sigma_{g,1}; \bR) \le H^1(UT \Sigma_{g,1}; \bR)$, and $PD: H^1(\Sigma_{g,1}; \bR) \to H_1(\Sigma_{g,1}; \bR)$ is the Poincar\'e duality isomorphism.
\end{lem}

\begin{proof}
Since $A(f, \chi)$ is rational it suffices to prove the claim when $\chi$ is restricted to be positive real valued. We will let $A$ be this restriction in the remainder of the proof.  By Corollary \ref{C:PrelimCocycleResult}, $A$ is a cocycle in $Z^1(\Mod_{g,1}, H^1(\Sigma_{g,1}, \bR_{>0})^*)$. Using the exponential isomorphism between $\bR$ and $\bR_{>0}$, $A$ can be identified with an additive cocycle in $Z^1(\Mod_{g,1}, H^1(\Sigma_{g,1}, \bR)^*)$. 

By Theorem \ref{T:MainMCGCohomology}, there is an element $W \in H^1(UT\Sigma_{g,1}; \bR)$ such that there is an equality of cocycles $A = \overline{C(W)}$. For a fixed $f \in \Mod_{g,1}$ and $\chi \in H^1(\Sigma_{g,1}; \bR)$, we can compute $A(f, \chi) = \overline{C(W)}(f)(\chi)$ as follows: $(f^*-1)W$ lies in the subspace $H^1(\Sigma_{g,1}; \bR) \le H^1(UT \Sigma_{g,1}; \bR)$, and by Poincar\'e duality can be identified with $PD((f^*-1)W) \in H_1(\Sigma_{g,1}; \bR)$. Then
\[
A(f, \chi) = \chi(PD((f^*-1)W)).
\]

Recalling Proposition \ref{prop:framingchar}, the assertion that $W$ arises from a framing (i.e. $W = W_\xi$ for some framing $\xi$) then amounts to seeing that, one, $W$ is integrally-valued and two, that $W(\zeta) = 1$, where $\zeta$ is a positively-oriented fiber of the $S^1$-bundle $UT\Sigma_{g,1}$. The second condition is equivalent to specifying that the winding number of the positively oriented loop $\delta$ around the marked point of $\Sigma_{g,1}$ has winding number $2g-1$. 

Via the Johnson lift $\gamma \mapsto \hat \gamma$ (as defined in Section \ref{subsection:wnf}), for any simple closed curve $\gamma$, there is a well-defined winding number $W(\gamma) = W(\hat{\gamma}) \in \bR$. Moreover, if $T_\gamma$ is a Dehn twist about $\gamma$, then $T_\gamma^*W - W$ is contained in the subspace $H^1(\Sigma_{g,1}; \bR) \le H^1(UT\Sigma_{g,1}; \bR)$, and is computed there, via twist-linearity (Equation \eqref{eqn:twistlin}), as 
\[
T_\gamma^*W - W = W(\hat\gamma) PD(\gamma).
\]
 Therefore, $A(T_\gamma, \chi) = \chi(W(\gamma)\gamma) = \chi(\gamma)^{W(\gamma)}$. It may be useful to recall that $\chi \in H^1(\Sigma_{g,1}, \bR_{>0}) \cong H^1(\Sigma_{g,1}, \bR)$ and hence can be viewed, after taking a logarithm, as an $\bR$-linear homomorphism from $H_1(\Sigma_{g,1}, \bR)$ to $\bR$. Since $A(T_\gamma, \chi)$ is a rational function, it follows that $W(\gamma)$ is integral for every simple closed curve $\gamma$ and hence $W \in H^1(UT\Sigma_{g,1}, \bZ)$. 
 
Suppose that $\gamma$ is a element of $\pi_1$ based at the marked point $p$ of $\Sigma_{g,1}$. Let $T$ be a tubular neighborhood of $\gamma$ with boundary components $\gamma_1$ and $\gamma_2$. Choose $\gamma_1$ so that an arc that starts on $\gamma_1$, crossing $\gamma$ once and ending on $\gamma_2$ will intersect $\gamma$ positively. A point push $f$ around $\gamma$ is then $f = T_{\gamma_2}^{-1} T_{\gamma_1}$. By twist linearity, 
\[ A(f, \chi) = \chi \left( -(T_{\gamma_1}^*W)(\widehat{\gamma_2})\gamma_2 + W(\widehat{\gamma_1})\gamma_1 \right) = \chi(\gamma)^{W(\widehat{\gamma_1}) - W(\widehat{\gamma_2})}. \]
By Lemma \ref{L:PointPush}, $W(\widehat{\gamma_1}) - W(\widehat{\gamma_2}) = 2g-2$. Since $W(\widehat{\gamma_1}) = W(\widehat{\gamma_2}) + W(\widehat{\delta}) - 1$, the result follows.
\end{proof}

\subsection{Proofs}
We proceed to the proof of the first half of Theorem \ref{T:Main:InvariantSection}, whose statement we recall:\\

\noindent \textbf{Theorem \ref{T:Main:InvariantSection} (First Half).} \emph{Fix $g \geq 3$. Let $\Gamma \leq \Mod_{g,1}$ project to a finite index subgroup of $\mathrm{Sp}(2g, \bZ)$ under the action on homology. Then $\bigwedge^{2g-1}{\cV_\bR}^*$ admits a $\Gamma$-equivariant non-vanishing section if and only if there is an $N$-framing $\xi$, for some positive integer $N$, so that $\Gamma \leq \Mod_{g,1}[\xi]$.}\\

\begin{proof}[Proof of Theorem \ref{T:Main:InvariantSection} (First Half):] 
Suppose first that $\Gamma$ preserves an $N$-framing $\xi$. By Lemma \ref{L:ItsACoboundary}, there is a framing $\eta$ so that $A(f, \chi) = \chi((f^*-1)W_\eta)$. There is some $v \in H^1(\Sigma_{g,1}, \bZ)$ so that $W_\xi = NW_\eta + v$. Thus
\[
A(f, \chi) = \chi((f^*-1)W_\eta) = \chi((f^*-1)\tfrac{1}{N}(W_\xi - v)) = \chi((f^*-1)\tfrac{-1}{N}v),
\]
exhibiting the restriction of $A$ to $\Gamma$ as a coboundary. This establishes the backwards implication. 

Suppose now that $\Gamma \leq \Mod_{g,1}$ acts on $H^1(\Sigma_g)$ by a finite-index subgroup $\Lambda$ of $\mathrm{Sp}(2g, \bZ)$. Suppose that there is a $\Gamma$-equivariant section of $\bigwedge^{2g-1}{\cV_\bR}^*$. Since the Torelli group $\cI_{g,1}$ preserves every point in the base space $H^1(\Sigma_g; \bR_{>0})^\circ$ and acts on the fiber over $\chi$ via multiplication by $A(f,\chi)$, it follows that the restriction of $A(f, \chi)$ to $\cI_{g,1}$ must be trivial.  We must show that $\Gamma$ belongs to an $N$-framed mapping class group for some $N$.

We will begin by showing that $A(f, \chi)$ is a coboundary when restricted to $\Gamma$. By the inflation-restriction sequence, the following sequence is exact:
\[ 0 \ra H^1(\Lambda, H_1(\Sigma_g, \bR_{>0})) \ra H^1(\Gamma, H_1(\Sigma_g, \bR_{>0})) \ra H^1(\cI_{g,1} \cap \Gamma, H_1(\Sigma_g, \bR_{>0}))^{\Lambda}.  \]
We will show that $H^1(\Lambda, H_1(\Sigma_g, \bR_{>0})) = 0$. By the Borel vanishing theorem \cite{Borel-Vanishing} (see also Tshishiku \cite{Tshishiku-Borel}),  since $\Lambda$ is finite index and since $g \geq 3$, $H^1(\Lambda, H_1(\Sigma_g; \bR)) = 0$ and this is isomorphic to $H^1(\Lambda, H_1(\Sigma_g, \bR_{>0}))$ using the exponential isomorphism between $\bR$ and $\bR_{>0}$.

Therefore, $A(f, \chi)$ is a coboundary when restricted to $\Gamma$ if and only if $A(f, \chi)$ is trivial for $f \in \Gamma \cap \cI_{g,1}$. As noted above, under our hypotheses, $A(f, \chi)$ is trivial when $f \in \Gamma \cap \cI_{g,1}$, and so $A(f, \chi)$ is a coboundary when restricted to $\Gamma$.

By Lemma \ref{L:ItsACoboundary}, there is a framing $\xi$ such that $A(f,\chi) = \chi(f^*W_\xi - W_\xi)$. 
Since the restriction of $A$ to $\Gamma$ is a multiplicative coboundary, there is a cohomology class $v \in H^1(\Sigma_{g,1}; \bR)$ so that 
\[
\chi(f^*(W_\xi+v)-(W_\xi+v)) = 0
\]
for all $\chi \in H^1(\Sigma_{g,1}, \bR_{>0})$ and all $f \in \Gamma$. Hence the class $W_\xi + v \in H^1(UT\Sigma_{g,1}; \bR)$ is $\Gamma$-invariant. It remains to show that $v$ is a rational cohomology class, since then an integer multiple of  $W_\xi + v$ will determine an integral $N$-framing for some positive integer $N$.

Since $W_\xi$ is integral, it follows that for the finite index subgroup $\Lambda$ of $\mathrm{Sp}(2g, \bZ)$ and for all $\gamma \in \Lambda$, $\gamma^*v - v \in H^1(\Sigma_{g,1}; \bZ)$. In particular, $\Lambda$ contains a $\gamma$ that has no eigenvalues equal to $1$, from which it follows that $(\gamma^* - I)$ is invertible and hence that $v \in H^1(\Sigma_{g,1}; \bQ)$ as desired. 
\end{proof}


We recall the following result, which is due to Ghazouani \cite[Proof of Theorem 5.1]{Ghazouani-MCGDynamicsAndHolonomy}.

\begin{thm}\label{T:JohnsonImage}
For generic $\chi \in H^1(\Sigma_{g,1}, \bR_{>0})$, the image of the Johnson kernel in $\mathrm{SL}\left( H^1(\Sigma_{g,1}, \bR_\chi) \right)$ is dense. 
\end{thm}
Using the isomorphism $H^1(\Sigma_{g,1}, \bR_{>0}) \cong H^1(\Sigma_{g,1}, \bR)$, ``generic" means belonging to the complement of the union of finitely many subspaces.
\begin{proof}
We will explain how to deduce this result from the work of Ghazouani. Ghazouani states this result for $g=2$ in Proposition 4.1 and in the proof of Theorem 5.2 for genus $g > 2$. Moreover, he only states the result for the image of $\cI_{g,1}$ in $\mathrm{PSL}\left( H^1(\Sigma_{g,1}, \bR_\chi) \right)$. However, his proof only uses Dehn twists about separating curves, which are elements of the Johnson kernel $\cK_{g,1}$. To move from $\mathrm{PSL}$ to $\mathrm{SL}$, we note that by Proposition \ref{prop:kginframed}, $\cK_{g,1}$ is contained in every framed mapping class group subgroup of $\Mod_{g,1}$. Thus by Lemma \ref{L:ItsACoboundary}, the function $A(\cdot, \chi)$ annihilates $\cK_{g,1}$ and hence the action of $\cK_{g,1}$ lifts to $\mathrm{SL}\left( H^1(\Sigma_{g,1}, \bR_\chi) \right)$. Therefore, it has dense image in $\mathrm{SL}$ since the map from $\mathrm{SL}$ to $\mathrm{PSL}$ is finite-to-one.
\end{proof}

We next recall and prove the second half of Theorem \ref{T:Main:InvariantSection}. Recall that a subgroup $\Gamma \le \Mod_{g,1}$ is {\em sufficiently large} 
if it contains the Johnson kernel $\cK_{g,1}$ and surjects onto a finite-index subgroup of $\mathrm{Sp}(2g,\bZ)$. Also recall the trivial line subbundle $L \subset \cV_\bR^+$ spanned by the section $\sigma$ given by
\[
\sigma(\chi)=  \chi\, z + (1 - \chi).
\]
Passing from $\cV_\bR^+$ to $\cV_\bR^+/L$ reduces the fiber over $\chi$ from $Z^1(\Sigma_{g,1}, \bR_\chi)$ to $H^1(\Sigma_{g,1}, \bR_\chi)$ and hence allows us to apply Theorem \ref{T:JohnsonImage}.

\noindent \textbf{Theorem \ref{T:Main:InvariantSection}  (Second Half). } {\em Fix $g \geq 3$.  If $\Gamma \leq \mathrm{Mod}_{g,1}$ is sufficiently large, then $\cV^+_\bR/L$ admits a $\Gamma$-invariant Lebesgue class measure if and only if $\Gamma$ is contained in an $N$-framed mapping class group for some $N$. When this occurs, the measure is unique up to scaling and is ergodic.}

\begin{proof}[Proof of Theorem \ref{T:Main:InvariantSection} (Second Half): ]
We will prove the first statement. The reverse direction follows from the first half of Theorem \ref{T:Main:InvariantSection}, which constructs an equivariant family of measures on the fiber of $\cV_\bR^+/L$, when $\Gamma$ is contained in an $N$-framed mapping class group. The resulting (fibered) product measure is invariant.

For the forward direction, suppose that $\Gamma$ is a subgroup of $\Mod_{g,1}$ that contains $\cK_{g,1}$, projects to a finite index subgroup of $\mathrm{Sp}(2g, \bZ)$, and that preserves a measure $\mu$ on $\cV_\bR^+/L$. Let $\nu$ be the measure defined by the symplectic form on $H^1(\Sigma_{g,1}, \bR_{>0}) \cong H^1(\Sigma_{g, 1}, \bR)$. By the disintegration theorem, there is a $\Gamma$-equivariant family of measures $\mu_\chi$ on $H^1(\Sigma_{g,1}, \bR_\chi)$ so that $\mu = \mu_\chi \otimes \nu$. Each measure is $\nu$-almost surely invariant under the action of $\cK_{g,1}$. By Theorem \ref{T:JohnsonImage} these measures are $\mathrm{SL}(H^1(\Sigma_{g,1}, \bR_\chi))$-invariant and hence multiples of Lebesgue measure on the fiber. The argument in the proof of the first half of Theorem \ref{T:Main:InvariantSection} completes the proof. 

Turning to ergodicity, suppose that $\Gamma$ is contained in an $N$-framed mapping class group. Let $\mu$ be the fully supported Lebesgue measure that is $\Gamma$-invariant. If $A$ is a measurable invariant set then $\mu' := 1_A \mu$ is another invariant measure where $1_A$ is the indicator function of $A$. Suppose that $\mu'$ is not the zero measure, i.e. suppose that $A$ is not null. Disintegrate as before into a $\Gamma$-equivariant family of measures $\mu_\chi'$. As before, each is a multiple of Lebesgue measure on the fibers. Since $\Gamma$ acts ergodically on $H^1(\Sigma_{g,1}, \bR)$, $\mu_\chi'$ is $\nu$-almost surely nonzero and hence $A$ is conull. 
\end{proof}

Finally, we recall and prove Theorem \ref{T:Main:InvariantForms}.

\noindent \textbf{Theorem \ref{T:Main:InvariantForms}. } {\em If $g \geq 3$ and $\Gamma \leq \mathrm{Mod}_{g,1}$ is sufficiently large, then $\mathrm{Hom}^\circ(\pi_1, \Dil)$ admits a $\Gamma$-invariant Lebesgue class measure if and only if $\Gamma$ is contained in an $N$-framed mapping class group for some $N$.}

\begin{proof}[Proof of Theorem \ref{T:Main:InvariantForms}:] The reverse direction follows from Theorem \ref{T:Main:InvariantSection}. For the forward direction, we use the decomposition $\mathrm{Hom}^\circ(\pi_1, \Dil) = (L \oplus(\cV_\bR^+/L)) \otimes \bC$. The action of the framed mapping class group on $L$ is trivial. Since an invariant Lebesgue class measure is equivalent to an equivariant volume form on the fibers, its existence is equivalent to the existence of an equivariant volume form on the fibers of $(\cV_\bR^+/L) \otimes \bC$. Let $\mu_\chi$ be this family of volume forms. Suppose that $V$ is a fiber of the bundle $\cV_\bR^+/L$. By Theorem \ref{T:JohnsonImage}, the Johnson kernel acts on $V$ by a dense subgroup of $\mathrm{SL}(V)$. When $\dim V > 2$, equivalently, $g > 2$, the diagonal action of $\mathrm{SL}(V)$ on $V \oplus V$ is transitive away from the null set of collinear pairs of vectors. This shows that $\mu_\chi$ is almost surely a multiple of Lebesgue measure on fibers. We now conclude as in the proof of Theorem \ref{T:Main:InvariantSection}.
\end{proof}

\color{black}

\section{Dilation surfaces}\label{S:DilationSurfaces}

\subsection{Basics} A \emph{dilation surface} structure on $\Sigma_{g,n}$ is an atlas of charts to $\bC$ with transition functions given by maps of the form $az+b$ where $a \in \bR_{>0}$ and $b \in \bC$  and so that, in a neighborhood of any puncture, the connection defined by pulling back the connection $d$ on $\bC$ under coordinate charts, determines a regular holomorphic connection after filling in the puncture. This final condition is equivalent to specifying local models for the dilation surface in neighborhoods of punctures. 

By its definition, a dilation surface $X$ carries a conformal structure. When speaking of dilation surfaces, $X$ will be a compact surface and $C$ will be the punctures, which are also called \emph{cone points}. Parallel transport produces a \emph{holonomy character} $\chi: \pi_1(X\setminus C) \ra \bR_{>0}$. When $\chi$ is trivial, $X$ is said to be a {\em translation surface}. Let $K_X$ be the canonical bundle and $L_\chi$ the flat line bundle over $X \setminus C$ defined by $\chi$. A dilation surface is equivalent to the data of $(X, C, \chi, \omega)$ where $X$ is a Riemann surface, $C$ a finite collection of points, $\chi$ a holonomy character, and $\omega$ a meromorphic section (defined up to scaling by positive real numbers) of $K_X \otimes L_\chi$ whose only zeros and poles occur at $C$ (see Apisa-Bainbridge-Wang \cite[Section 2]{ABW1}). This statement makes precise the claim that a dilation surface is equivalent to a twisted holomorphic $1$-form up to scaling. 

The holonomy of a positively-oriented loop around a puncture is called the \emph{local holonomy}. It may be written uniquely as $e^{-2\pi b}$ for $b \in \bR$. Write $2\pi |a|$ for the cone angle of a cone point. The \emph{signed cone angle}, written $2\pi a$, is then given by $a = |a|$ if the puncture is a zero and $a = -|a|$ if it is a pole. The \emph{complex cone angle} of the puncture $c$ is then $\kappa_c := a+ib \in \bZ + i \bR$. The negative complex cone angle is the residue at $c$ of the regular connection determined by the dilation surface. Finally, part of the data of $\kappa$ will be whether each pole with trivial local holonomy has vanishing or nonvanishing residue. We say that the dilation surface belongs to the \emph{stratum} $\cD(\kappa)$ where $\kappa = (\kappa_c)_{c \in C}$.


\subsection{Scalings and twisted holomorphic $1$-forms}\label{SS:scalings} If $X$ is a dilation surface endowed with no other structure, there is not a natural choice of scale: precisely, $X$ and $\mathrm{diag}(\lambda, \lambda) \cdot X$ are isomorphic as dilation surfaces and so determine the same point of $\cD(\kappa)$ for any $\lambda \in \bR_{>0}$. Thus the $\mathrm{GL}(2, \bR)$ action reduces to one by $\mathrm{SL}(2, \bR)$. We will say that two dilation surfaces that differ by the action of $\mathrm{SO}(2)$ are \emph{the same up to rotation}. We will write the set of dilation surfaces in $\cD(\kappa)$ up to rotation as $\cD(\kappa)/S^1$. In particular, $\cD(\kappa) \ra \cD(\kappa)/S^1$ is a circle bundle. Let $\cD_s(\kappa)$ be the complement of the zero section of the associated complex line bundle. One advantage to working with $\cD_s(\kappa)$ is that it is the complement of the zero section for the pullback of a holomorphic line bundle over $\cM_{g,n}$ and hence a complex analytic space, unlike $\cD(\kappa)$. 

 When $X$ has a cone point $c$ with trivial local holonomy, the elements of $\cD_s(\kappa)$ can be understood geometrically. In this case, there is a holomorphic $1$-form $\omega_c$, defined up to positive real scaling, in a neighborhood $U_c$ of $c$ that gives the dilation surface structure. The elements of $\cD_s(\kappa)$ are simply the elements of $\cD(\kappa)$ together with a \emph{scaling}, i.e. the germ at $c$ of some choice of $\omega_c$. It is clear that the $\mathrm{SL}(2, \bR)$ action on $\cD(\kappa)$ extends to a $\mathrm{GL}(2, \bR)$ action on $\cD_s(\kappa)$. If $\omega_c$ has a zero at $c$, then we will call $c$ \emph{translation-like}. The name is chosen since $c$ has a neighborhood that admits the same local model as that of a cone point on a finite area translation surface. 

Let $T$ be the collection of cone points with trivial local holonomy on $X$. The dilation surface structure on $X$ is specified by a holomorphic $1$-form $\omega$, up to positive real scaling, on the universal cover of $X \setminus (C \setminus T)$ so that $\gamma^* \omega = \chi(\gamma)^{-1} \omega$ for all $\gamma \in \pi_1(X \setminus (C \setminus T))$. This is called a twisted holomorphic $1$-form and its moduli space is $\Omega^{tw} \cM_{g,n}(\kappa)$. A scaling allows us to specify $\omega$ precisely, not just up to scaling, by specifying a lift $\wt{c}$ of a point $c \in T$ and demanding that $\omega$ agree with the lift of $\omega_c$ in a neighborhood of $\wt{c}$. We record the following consequence of this discussion.

\begin{lem}
Suppose that $\cD(\kappa)$ is a stratum of genus $g$ dilation surfaces with at least one cone point with trivial local holonomy and exactly $n$ points with nontrivial local holonomy. Then there is a bijection between $\cD_s(\kappa)$ and $\Omega^{tw}\cM_{g,n}(\kappa)$.
\end{lem}

\subsection{Scalings and prongings in the presence of translation-like cone points}

Our description of period coordinates in the next subsection will depend on using a translation-like cone point. We will therefore give a second description of the scalings in the setting where such cone points are present. A translation-like cone point can always be formed by marking a point.

Let $X \in \cD(\kappa)/S^1$ be a dilation surface up to rotation. Suppose that it has a translation-like cone point $p$ with cone angle $2\pi k > 0$. In particular, this determines a flat metric up to scaling in a neighborhood of $p$. A preimage of $X$ under the map $\cD(\kappa) \ra \cD(\kappa)/S^1$ is determined by a {\em pronging} at $p$, i.e. by $k+1$ distinct evenly-spaced unit tangent vectors at $p$. The pronging resolves the ambiguity of which directions are horizontal. The relative tangent bundle $L$ at $p$ over $\cD(\kappa)/S^1$ is the bundle whose fiber over $X$ is $T_pX$. A pronging is a section of the circle bundle associated to $L^{\otimes (k+1)}$. In the terminology of the tautological ring of $\cM_{g,n}$, the Euler class of the circle bundle $\cD(\kappa) \ra \cD(\kappa)/S^1$ is the pullback of $-(k+1)\psi_p$ where $\psi_p$ is the $\psi$-class associated to $p$.

Say that a \emph{scaled pronging} or, more briefly, a \emph{scaling}, is a collection of $k+1$ distinct evenly spaced nonzero tangent vectors of equal length at $p$. These vectors are no longer required to be unit vectors. Given a dilation surface $X$ up to rotation, a scaling tells us both which direction is horizontal and provides a way to choose a scale for the flat metric that the dilation structure induces on a neighborhood of $p$. The next result shows that the new definition of a scaling agrees with the one defined in the previous section. The surfaces in $\cD_s(\kappa)$ are \emph{scaled dilation surfaces}, i.e. dilation surfaces up to rotation together with a scaling at $p$. We stress that scaled dilation surfaces come equipped with a choice of a translation-like cone-point $p$. We note that $\mathrm{GL}(2, \bR)$ acts on $\cD_s(\kappa)$ with the constant multiples of the identity changing the scaling at $p$ and hence now acting nontrivially. 

\begin{lem}
Given a dilation surface $X$ and a translation-like cone point $p$, there is a flat metric defined on a neighborhood of $p$ up to scaling by a positive-real scalar. The choice of specific flat metric (no longer just up to scaling) is equivalent to choosing a scaling at $p$.
\end{lem}
\begin{proof}
Suppose that $v$ is a tangent vector at $p$. Using local coordinates we will identify $p$ with the origin and therefore can use $tv$ as a smooth path through $p$. Converting to flat coordinates is then formed by applying the map $f(z) = z^{k+1}$. Therefore, our smooth path becomes $t^{k+1} v^{k+1}$ in flat coordinates. We choose the flat metric so that this path has distance $t^{k+1}$ from $p$. Therefore, in a coordinate-free formulation, suppose that $\gamma: (-\epsilon, \epsilon) \ra X$ is a smooth path so that $\gamma(0) = p$ and $\gamma'(0) = v$. Then we may uniquely choose the flat metric in a neighborhood of $p$ so that $\mathrm{dist}(p, \gamma(t)) = t^{k+1} + O(t^{k+2})$. Conversely, we see that by setting $|v| := \lim_{t \ra 0} \frac{\mathrm{dist}(p, \gamma(t))}{t^{k+1}}$, a flat metric determines a function $|\cdot |: T_p X \ra \bR_{\geq 0}$ that is $\mathrm{SO}(2)$ invariant and so that $|rv| = r^{k+1}|v|$ for any real positive $r$.
\end{proof}

\subsection{Holonomy and periods}\label{SS:monholo}

As in the classical setting of translation surfaces, a dilation surface has an associated {\em period mapping}. To formulate this, we must understand how to interpret a dilation surface as a complex-analytic object.

A dilation surface with holonomy character $\chi$ is equivalent to a Riemann surface structure $X$ on $\Sigma_{g,n}$ together with a holomorphic $1$-form $\omega$ on its universal cover so that $\gamma^* \omega = \chi(\gamma)^{-1} \omega$ for all $\gamma \in \pi_1(\Sigma_{g,n})$ where $\omega$ is determined up to positive real scaling. If the dilation surface has a translation-like cone point $p$ specified and a choice of a lift $\wt{p}$ of $p$ to the universal cover, then a scaling at $p$ specifies a scale factor for $\omega$. In other words, it picks out a specific $\omega$ from the set $\{c\omega\}_{c > 0}$. One could also think of $\omega$ as a section of the twisted canonical bundle $K_X \otimes L_\chi$ where $K_X$ is the canonical bundle and $L_\chi$ is a flat holomorphic bundle with character $\chi$. This leads to the construction of the period map for dilation surfaces: define 

 \[
 \lambda: \pi_1(\Sigma_{g,n}, p) \ra \bC
 \]
 via
 \[
 \lambda(\gamma) = \int_{\wt{\gamma}} \omega
 \] 
 where $\wt{\gamma}$ is the lift of $\gamma$ so that it is based at $\wt{p}$.  It is easy to see that $\lambda$ is a twisted cocycle for the coefficient module $\bC_\chi$, i.e. $\bC$ with $\pi_1(\Sigma_{g,n})$ acting multiplicatively by $\chi$:
\begin{equation}\label{eqn:cocycle2}
    \lambda(\gamma_2\gamma_1) = \lambda(\gamma_2) + \chi(\gamma_2)\lambda(\gamma_1).
\end{equation}

As before we consider the space $Z^1(\Sigma_{g,n}; \bC_\chi)$ of these cocycles and the associated cohomology groups $H^1(\Sigma_{g,n}; \bC_\chi)$.

Define the {\em dilation subgroup} $\Dil \le \Aff(\bC)$ via
\[
\Dil = \{az+b \mid a \in \bR_{>0}, b \in \bC\}.
\]

An unscaled dilation surface can be described as the data of a $(\Dil, \bC)$-structure on $\Sigma_{g,n}$. In general, the developing map associates to a dilation surface a homomorphism $\pi_1(\Sigma_{g,n}, p) \ra \Dil$ up to conjugation by $\Dil$. The developing map for a scaled dilation surface removes the need for this conjugation. We remark that a homomorphism from $\pi_1(\Sigma_{g,n}, p) \ra \Dil$ has the form $\gamma \ra \chi(\gamma) + \lambda(\gamma)$, for each $\gamma \in \pi_1(\Sigma_{g,n}, p)$, where $\chi \in H^1(\Sigma_{g,n}, \bR_{>0})$ and where $\lambda \in Z^1_\chi(\bC)$. We conclude with the following crucial observations that relate the moduli of dilation surfaces to the bundles studied in the previous sections.

\begin{lem}
The developing map associates a scaled dilation surface to an element of $\mathrm{Hom}(\pi_1(\Sigma_{g,n}, p), \Dil)$ that is given as $\chi(\gamma)z + \lambda(\gamma)$, for all $\gamma \in \pi_1(\Sigma_{g,n}, p)$ where $\chi \in H^1(\Sigma_{g,n}; \bR_{>0})$ is the holonomy character and $\lambda \in Z^1_\chi(\bC)$ is the period of the dilation surface.
\end{lem}

As in the previous section, there is a derivative map 
\[
\Pi: \mathrm{Hom}(\pi_1(\Sigma_{g,n}, p), \Dil) \to H^1(\Sigma_{g,n}; \bR_{>0}),
\]
and we define (using temporary notation) 
\[
\mathrm{Hom}^\circ(\pi_1(\Sigma_{g,n}, p), \Dil) := \Pi^{-1}(H^1(\Sigma_{g,n}; \bR_{>0})^\circ);
\]
in general the superscript $\circ$ indicates that the fiber over $0$ has been deleted.

\begin{lem}
    $\mathrm{Hom}^\circ(\pi_1(\Sigma_{g,n}, p), \Dil)$ is a complex vector bundle over $H^1(\Sigma_{g,n}, \bR_{>0})^\circ$, realizable as the complexification of the real vector bundle $\cV_{\bR}^+ = \mathrm{Hom}^\circ(\pi_1(\Sigma_{g,n}, p), \mathrm{Aff}^+(\bR))$.
\end{lem}

Following this, we re-define
\[
\cV_\bR^+ \otimes \bC :=\mathrm{Hom}^\circ(\pi_1(\Sigma_{g,n}, p), \Dil)
\]

$\cV_\bR^+ \otimes \bC$ admits an action of $\Mod_{g,n+1} \times \mathrm{SL}(2, \bR)$ where $\Mod_{g,n+1}$ acts by precomposition and $\mathrm{SL}(2, \bR)$ acts by post-composition on elements of $Z^1_\chi(\bC)$, viewed as functions from $\pi_1(\Sigma_{g,n},  p)$ to $\bC \cong \bR^2$. We then obtain the following construction of measures on $\cV_{\bR}^+ \otimes \bC$.

\begin{lem}\label{lemma:taco}
Any measure on $\cV_\bR^+$ whose restriction to each fiber is a multiple of Lebesgue measure determines an $\mathrm{SL}(2, \bR)$-invariant measure on $\cV_\bR^+ \otimes \bC$. 
\end{lem}

\subsection{Dilation surfaces and framings}
A dilation surface $X$ carries a distinguished isotopy class of framing $\xi$ induced from the vector field $\xi_1$ pointing in the direction of the oriented line-field $\ker(\Re \omega)$, where $\omega$ is the $1$-form associated to $X$.

The presence of the distinguished framing on a dilation surface has the following important consequence. The general theory of moduli provides for a {\em global monodromy homomorphism}
\[
\mu: \pi_1^{orb}(\cD(\kappa)/S^1) \to \Mod_{g,n}
\]
defined (somewhat informally) by recording the change-of-marking map as one moves around a loop in $\cD(\kappa)/S^1$. Choosing a basepoint $X \in \cD(\kappa)/S^1$ and letting $\xi$ be the associated distinguished framing, it follows that 
\[
\Im(\mu) \le \Mod_{g,n}[\xi].
\]
We note in passing that \cite{ABW1} establishes a vastly stronger result, that not only is this containment an equality, but in fact $\cD(\kappa)$ itself is an orbifold $K(\pi,1)$ space for this framed mapping class group. However, for the sequel, only the containment above will be necessary.

\section{Proof of Theorem \ref{T:main:general}}\label{S:ProofOfTheorems}

Suppose that $\cD(\kappa)$ is a stratum of dilation surfaces with at least one translation-like cone point $p$. Let $P$ be the set of translation-like cone points, $T$ the set of poles with trivial local holonomy and zero residue, and $B$ the remaining cone points. We will require a slight generalization of the notion of a twisted cocycle. The fundamental groupoid, notated $\pi_1(\Sigma_g \setminus B, P)$, is the category whose objects are $P$ and whose morphisms are arcs in $\Sigma_g \setminus B$ joining them, up to homotopies that fix endpoints. We will say that a \emph{relative twisted cocycle} $\lambda \in Z^1(\Sigma_g \setminus B, P; \bC_{\wt{\chi}})$ is a function $\lambda: \pi_1(\Sigma_g \setminus B, P) \ra \bC$ together with a character $\wt{\chi} \in H^1(\Sigma_g \setminus B, P; \bC^\times)$ so that the cocycle condition \eqref{eqn:cocycle} holds for composable paths. Label the points in $P$ as $\{p_0, p_1, \hdots, p_n\}$ where $p = p_0$. Let $c_1, \hdots, c_n$ be a collection of arcs from $p$ to $p_1, \hdots, p_n$ respectively.

\begin{lem}\label{L:ChoosingRelativeTwisted}
Fix $\lambda \in Z^1(\Sigma_g \setminus B; \bC_\chi)$.  Let $a_1, \hdots, a_n$ be in $\bC$ and $b_1, \hdots, b_n$ in $\bC^\times$. There is a unique character $\wt{\chi} \in H^1(\Sigma_g \setminus B, P; \bC^\times)$ that agrees with $\chi$ on $\pi_1(\Sigma_g \setminus B, p)$ and that sends $c_i$ to $b_i$ and there is a unique $\wt{\lambda} \in Z^1(\Sigma_g \setminus B, p; \bC_{\wt{\chi}})$ that agrees with $\lambda$ on paths in $\pi_1(\Sigma_g \setminus B, p)$ and that sends $c_i$ to $a_i$ for all $i$.
\end{lem}
\begin{proof}
Every path from $p$ to $p_i$ can be written uniquely as $\gamma c_i$ where $\gamma \in \pi_1(\Sigma_g \setminus B, p)$. Set $\wt{\lambda}(\gamma c_i) = \lambda(\gamma) + \chi(\gamma) a_i$. Notice that we must have 
\[
0 = \tilde \lambda( c_i c_i^{-1}) = \lambda(c_i) + \chi(c_i)\lambda(c_i^{-1}) = a_i+ b_i \lambda(c_i^{-1}).
\]
and so $\lambda(c_i^{-1}) = -b_i^{-1} a_i$.  Any path in $\pi_1(\Sigma_g \setminus B, P)$ can be written as $c_i^{-1} \gamma c_j$ where $\gamma \in \pi_1(\Sigma_g \setminus B, P)$.  We must define,
\[ \lambda(c_i^{-1} \gamma c_j) = -b_i^{-1}a_i + b_i^{-1} \lambda(\gamma) + b_i^{-1}\chi(\gamma) a_j. \]
It is straightforward to check that $\wt{\lambda}$ satisfies the cocycle condition.
\end{proof}

Now fix $\chi \in H^1(\Sigma_g \setminus B; \bC^\times)$. Let $\wt{\chi}$ be any character in $H^1(\Sigma_g \setminus B, P; \bC^\times)$ whose image in $H^1(\Sigma_g \setminus B; \bC^\times)$ is $\chi$. A \emph{based twisted cocycle} is the restriction of an element of $Z^1_{\wt{\chi}}(\Sigma_g \setminus B, P; \bC)$ to paths in $\pi_1(\Sigma_g \setminus B, P)$ that are based at $p$. By the proof of the preceding lemma, a based twisted cocycle $\wt{\lambda}$ is independent of the choice of $\wt{\chi}$ and is entirely determined by its restriction to an element of $Z^1_\chi(\Sigma_g \setminus B, p; \bC)$ and the values it takes on $c_1, \hdots, c_n$. Let $Z^1_\chi(\Sigma_g \setminus B, P, p; \bC)$ be the space of based twisted cocycles. Let $\mathrm{pr}: Z^1_\chi(\Sigma_g \setminus B, P, p; \bC) \ra Z^1_\chi(\Sigma_g \setminus B, p; \bC)$ be the projection that only remembers the values associated to closed loops based at $p$. 

\begin{lem}\label{L:SES:BasedToUsual}
There is a short exact sequence
\[ 0 \ra \bC^{|P|-1} \ra Z^1_\chi(\Sigma_g \setminus B, P, p; \bC) \xra{pr} Z^1_\chi(\Sigma_g \setminus B, p; \bC) \ra 0. \]
\end{lem}
\begin{proof}
This is immediate from Lemma \ref{L:ChoosingRelativeTwisted}. 
\end{proof}

This short exact sequence splits, since a given element of $Z^1_\chi(\Sigma_g \setminus B, p; \bC)$ has a unique extension to $Z^1_\chi(\Sigma_g \setminus B, P, p; \bC)$ given by demanding that each $c_i$ is sent to $0$. Let $(\xi_i)_{i=1}^n$ be the elements in $Z^1_\chi(\Sigma_g \setminus B, P, p; \bC)$ that restrict to zero on $\pi_1(\Sigma_g, p)$ and where $\xi_i(c_j) = \delta_{ij}$. Let $\gamma_i$ be the element of the dual space of $Z^1_\chi(\Sigma_g \setminus B, P, p; \bC)$ that evaluates a based twisted cocycle on $c_i$. 

Let $\cV^{rel}$ (resp. $\cV^{rel}_\bR$) be the bundle of complex (resp. real) based twisted cocycles over $H^1(\Sigma_g \setminus B; \bR_{>0})^\circ$. (Note that the characters in the base have real coefficients in either setting). Let $\cV^{rel}(\kappa)$ be the restriction of the bundle $\cV^{rel}$ to the subset of characters $\chi$ in $H^1(\Sigma_g \setminus B, \bR_{>0})$ where $\chi$ sends the loop around a point in $B$ to its local holonomy. 

If $\omega$ is a holomorphic $1$-form on the universal cover of $\Sigma_g \setminus B$ satisfying $\gamma^* \omega = \chi(\gamma)^{-1} \omega$ for all $\gamma \in \pi_1(\Sigma_g \setminus B, p)$, then the period map, which we continue to call $\lambda$, given by integrating paths from a lift $\wt{p}$ of $p$ to lifts of points in $P$, gives a based twisted cocycle as in Section \ref{SS:monholo}. This defines a {\em global period map} $\Phi: \wt{\cD_s}(\kappa) \ra \cV^{rel}(\kappa)$, where $\wt{\cD_s}(\kappa)$ is the pullback of $\cD_s(\kappa) \ra \cD(\kappa)/S^1$ to the universal cover of $\cD(\kappa)/S^1$, which is Teichm\"uller space.  We record the following result due to Apisa \cite{Apisa-Period}, extending work of Veech \cite{veech93}.

\begin{thm}\label{T:PeriodMapWithTransConePoints}
 The global period map is a framed mapping class group and $\mathrm{SL}(2, \bR)$ equivariant local diffeomorphism away from the locus of translation and homothety surfaces. 
 \end{thm}

To construct an invariant measure on a cover of $\cD_s(\kappa)$ we will construct a volume form on $\cV^{rel}$ that is invariant under a subgroup of the framed mapping class group and $\mathrm{SL}(2, \bR)$, use it to produce such a measure on $\cV^{rel}(\kappa)$, and pull this measure back to $\wt{\cD_s}(\kappa)$ under the global period map.



\begin{lem}\label{L:TrivialDeterminantBundle2}
The top exterior wedge of $(\cV^{rel})^*$ is a trivial line bundle.
\end{lem}
\begin{proof}
When the set $B$ of non-translation like cone points is empty we simply wedge $\gamma_1 \wedge \hdots \wedge \gamma_n$ with the section of $\bigwedge^{2g-1} {\cV}^*$ produced in Lemma \ref{L:TrivialDeterminantBundle}. 

When $B$ is nonempty, $\pi_1(\Sigma_g \setminus B)$ is a free group on $2g + |B| -1$ generators. Let $a_1, \hdots, a_{2g+|B|-1}$ be generators and let $(\alpha_i)_{i=1}^{2g+|B|-1}$ be the elements of $(\cV^{rel})^*$ formed by evaluating a based twisted cocycle on $(a_i)_{i=1}^{2g+|B|-1}$. The desired section is $\gamma_1 \wedge \hdots \wedge \gamma_n \wedge \alpha_1 \wedge \hdots \wedge \alpha_{2g+|B|-1}$. 
\end{proof}

As in Section \ref{S:Proof-OneConePoint}, given $\chi \in H^1(\Sigma_g \setminus B, \bC^\times)^\circ$, let $\omega_\chi$ be the volume form constructed on the fiber over $\chi$ in Lemma \ref{L:TrivialDeterminantBundle2}. For $f \in \mathrm{Mod}_{g, |B|+|T|+|P|}$, define the cocycle $A(f, \chi)$ to be the function 
\[
A(f, \cdot): H^1(\Sigma_g - B, \bC^\times)^\circ \ra \bC^\times
\]
so that
\[ f_* \omega_\chi = A(f, \chi) \omega_{f_* \chi}. \]
Similarly, define $\omega^{abs}_{\chi}$ to be the volume form on fibers produced in Lemma \ref{L:TrivialDeterminantBundle} when $B$ is empty and $\alpha_1 \wedge \hdots \wedge \alpha_{2g+|B|-1}$ otherwise. For $f \in \mathrm{Mod}_{g, |B|+1}$, define the cocycle $A^{abs}(f, \chi)$ to be the function so that
\[ f_* \omega^{abs}_\chi = A^{abs}(f, \chi) \omega^{abs}_{f_* \chi}. \]

\begin{lem}
$A^{abs}$ is a multiplicative cocycle and $A(f, \chi) = A^{abs}(f, \chi) \prod_{i=1}^n \chi(f(c_i) - c_i)$.
\end{lem}
We will refer to $A^{abs}$ as the \emph{absolute part of the cocycle} and $\prod_{i=1}^n \chi(f(c_i) - c_i)$ as the \emph{relative part}. 
\begin{proof}
When $B$ is empty the first claim is Corollary \ref{C:PrelimCocycleResult}. When $B$ is nonempty, the same proof applies, showing that $A^{abs}$ is multiplicative.

We will show the second claim in the case that $B$ is nonempty, with the first case being nearly identical. As in Lemma \ref{L:CocycleIsRational}, $A(f, \chi) = \det(\lambda_i(f(d_j)))_{i,j}$ where $(d_i)_i = (\gamma_i)_{i=1}^n \cup (a_i)_{i=1}^{2g+|B|-1}$ and where $\lambda_i$ is the dual basis of based twisted cocycles. By Lemma \ref{L:SES:BasedToUsual}, the mapping class group acts on fibers of $\cV^{rel}$ by a block upper triangular matrix consisting of two blocks on the diagonal, recording the actions on $\bC^{|P|-1}$ and $Z_\chi^1(\Sigma_g\setminus B, p; \bC)$, respectively. This shows that 
\[ A(f, \chi) = A^{abs}(f, \chi) \det\left( \xi_i(f(c_j)) \right)_{i,j=1}^n. \]

To evaluate $\det\left( \xi_i(f(c_j)) \right)_{i,j=1}^n$, write $f(c_i) = \eta_i c_i$ where $\eta_i \in \pi_1(X, p)$. Then,
\[ \xi_i(f(c_j)) = \xi_i(\eta_j) + \delta_{ij} \chi(\eta_i) = \delta_{ij} \chi(\eta_i)\] 
and the result follows.
\end{proof}


In final preparation for the proof of Theorem \ref{T:main:general}, we require the following result, Corollary \ref{C:zerorestrict}, proved in the Appendix. \vspace{0.5Em}

\noindent\textbf{Corollary \ref{C:zerorestrict}.} {\em 
    Let $\Sigma_g$ be a surface of genus $g \ge 2$, and let $Q = \{p_1, \dots, p_{n+1}\}$ be a set of marked points. Let $Q' = Q \setminus \{p_{n+1}\}$. 
    Let $G \le \Mod(\Sigma_g \setminus Q)$ be a subgroup. Suppose $G$ satisfies the following conditions:
    \begin{enumerate}
        \item There is a framing $\eta$ of $\Sigma_g \setminus Q$ for which $G \le \Mod(\Sigma_g \setminus Q)[\eta]$.
        \item There is a basis $\tau$ of arcs in $\wt{H}^0(Q)$  such that $g[\alpha_i] = [\alpha_i]$ in $H_1(\Sigma_g, Q)$ for all $g \in G$ and $\alpha_i \in \tau$.
        \item There is a non-vanishing vector field $\xi_1$ on $\Sigma_g \setminus Q$ with zero winding number about each $p_i \in Q'$ such that $W_{\xi_1}(g(\beta_i)) = W_{\xi_1}(\beta_i)$ for all $g \in G$ and for all arcs $\beta_i$ forming a basis of $\wt{H}_0(Q')$. 
    \end{enumerate}
    Then the restriction map $H^1(\Mod_{g,n+1}; H_1(\Sigma_g\setminus Q')) \to H^1(G; H_1(\Sigma_g\setminus Q'))$ is zero.
}


\begin{cor}\label{C:MainTool}
When $\Gamma$ is a subgroup of $\mathrm{Mod}_{g, |B| + |T|+|P|}$ described in Theorem \ref{T:main:general}, there is a $\Gamma$-equivariant nonvanishing section of the top exterior wedge of $\cV_\bR^{rel}$. In particular, there is a $\Gamma$-equivariant family of volume forms on the fibers of $\cV^{rel}$.
\end{cor}
\begin{proof}
The final claim will follow from the first one since $\cV^{rel}$ is the complexification of $\cV^{rel}_\bR$ (see Lemma \ref{lemma:taco}). 

The desired section exists if and only if the cocycle $A(f, \chi)$ restricts to $\Gamma$ as a coboundary (see the discussion after Lemma \ref{L:LinearAlgebra}). In the case that $B$ is empty, $\Gamma$ is the subgroup of a framed mapping class group that preserves each $c_i$ as an element of relative homology. The absolute part of the cocycle restricts to a coboundary by Theorem \ref{T:Main:InvariantForms} and the relative part of the cocycle restricts to a coboundary by its definition. 

When $B$ is nonempty, $\Gamma$ is defined as a subgroup of $\mathrm{Mod}_{g, |B|+|T|+|P|}$ preserving the following set of data:
\begin{enumerate}
    \item $\Gamma$ preserves the distinguished framing $\eta$ associated to a dilation surface.
    \item $\Gamma$ preserves the relative homology classes of $(c_i)_{i=1}^n$
    \item Let $\tau$ be a set of arcs connecting a fixed point in $B$ to each point in $P$. $\Gamma$ preserves the relative homology classes of $\tau$. 
    \item There is a nonvanishing vector field $\xi_0$ with vanishing winding number around points in $P$ and a basis $\tau'$ of $\wt{H}_0(P)$ so that $\Gamma$ preserves the relative winding numbers of elements in $\tau'$ as computed with respect to $\xi_0$. 
\end{enumerate}
The relative part of the cocycle restricts to a coboundary by the second condition. The absolute part of the cocycle is a change of relative winding number cocycle $\overline{C(\xi_1)}$ by Theorem \ref{T:MainMCGCohomology} for a nonvanishing vector field $\xi_1$ whose winding number around the points in $P$ vanishes. This restricts to a coboundary on $\Gamma$ by Corollary \ref{C:zerorestrict}.
\end{proof}

\begin{proof}[Proof of Theorem \ref{T:main:general}:]
Forming the product measure using the $\Gamma$-equivariant family of measures from Corollary \ref{C:MainTool} and the $\Gamma$-invariant Lebesgue measure on the base of $\cV^{rel}(\kappa)$ produces a $\Gamma$-invariant Lebesgue class measure on $\cV^{rel}(\kappa)$. It is $\mathrm{SL}(2, \bR)$-invariant by construction. By Theorem \ref{T:PeriodMapWithTransConePoints}, the pullback of this volume form under the period map defines a differential form that determines a $\Gamma \times \mathrm{SL}(2, \bR)$-invariant measure on $\wt{\cD_s}(\kappa)$. The form's $\Gamma$-invariance allows it to descend to the desired form on the quotient of $\wt{\cD_s}(\kappa)$ by $\Gamma$.
\end{proof}

\section{Proof of Theorem \ref{T:area}}\label{S:ProofOfAreaCondition}

Let $\cD_s \ra \cD$ be the real line bundle of scaled dilation surfaces over a stratum of unscaled dilation surfaces. By Theorem \ref{T:main:general}, after passing to an explicit cover $\wt{\cD}$ and letting $\wt{\cD_s}$ be the pullback of $\cD_s$, $\wt{\cD_s}$ admits an $\mathrm{SL}(2, \bR)$-invariant Lebesgue class measure $\nu$. After additionally passing to a double cover, one may assume that the bundle $\wt{\cD_s} \ra \wt{\cD}$ is trivial and hence that $\wt{\cD_s}$ can be identified with $\wt{\cD} \times \bR_{>0}$. Let $r$ be the projection to the $\bR_{>0}$ factor. Let $\mu$ be Lebesgue measure on $\wt{\cD}$. Up to altering $\mu$, disintegrating $\nu$ allows us to write it as $\mu \otimes r^n dr$ where $n$ is an integer. Finally, the action by $g \in \mathrm{SL}(2, \bR)$ on $\wt{\cD}_s$ sends a point $(x, c)$ to $(g\cdot x, A(g,x) \cdot c)$ where $A: \mathrm{SL}(2, \bR) \times \wt{\cD} \ra \bR_{>0}$ is a cocycle. Recall that $A(g, x)$ is a \emph{coboundary} if there is a measurable function $F: \cD \ra \bR_{>0}$ so that $\frac{F(g \cdot x)}{F(x)} = A(g,x)$ for all $g \in \mathrm{SL}(2, \bR)$ and $x \in \wt{\cD}$.

\begin{lem}
There is an $\mathrm{SL}(2, \bR)$ invariant measure on $\wt{\cD}$ if and only if $A$ is a coboundary.
\end{lem}
\begin{proof}
Since $\nu$ is invariant, $\frac{d(g_*\mu)}{d\mu}A(g, x) = 1$ for all $g \in \mathrm{SL}(2, \bR)$ and $x \in \wt{\cD}$. An invariant measure on $\wt{\cD}$ exists if and only if the Radon-Nikodym cocycle $\frac{d(g_*\mu)}{d\mu}$ is a coboundary. This occurs if and only if $A$ is a coboundary. 
\end{proof}

If we fix any norm $| \cdot |$ on the fibers of $\wt{\cD_s} \ra \wt{\cD}$, then if $X$ is any preimage of $x \in \wt{\cD}$ on $\wt{\cD}_s$, then $A(g,x) = \frac{|g \cdot X|}{|X|}$. Without loss of generality this norm is smooth and $\mathrm{SO}(2)$-invariant. Say that an \emph{area-function} is a map $a: \wt{\cD}_s \ra \bR_{>0}$ that is $\mathrm{SL}(2, \bR)$-invariant and so that for any real number $\lambda$, $a(\lambda \cdot x) = \lambda^2 a(x)$. The following statement is a stronger form of Theorem \ref{T:area}.

\begin{thm}
There is an $\mathrm{SL}(2, \bR)$ invariant measure on an open invariant subset of $\wt{\cD}$, if and only if there is an area function on the subset. 
\end{thm}
\begin{proof}
If $F$ is as in the definition of ``coboundary", then $\frac{|X|^2}{F(X)^2}$ is the desired area function. Conversely, if $a$ is the area function, then set $F(X) = \frac{|X|}{\sqrt{a(X)}}$.
\end{proof}

\begin{cor}\label{C:ReebLocus}
Consider the open subset of $\wt{\cD}$ consisting of surfaces with Reeb tori not all of whose boundary saddle connections are parallel. This subset is equipped with an $\mathrm{SL}(2, \bR)$ invariant measure. 
\end{cor}
\begin{proof}
The typical surface in this locus contains Reeb tori with simple boundaries. Let $v$ and $w$ be the periods of any two generically non-parallel such boundary saddle connections. Then $|v \wedge w|$ is the desired area function. 
\end{proof}

\begin{rem}
Up to null sets, the complement of the open subset of $\wt{\cD}$ considered in Corollary \ref{C:ReebLocus} is the triangulable locus. So the issue of whether $\wt{\cD}$ has an $\mathrm{SL}(2, \bR)$-invariant fully supported Lebesgue class measure comes down to whether the triangulable locus admits an area function.
\end{rem}

\begin{rem}
The strategy to produce an invariant measure for scaled dilation surfaces will not work in the unscaled case, at least not in the case where all singularities are translation like. To see this, we note that Ghazouani \cite{Ghazouani-MCGDynamicsAndHolonomy} has shown that the Johnson kernel, which belongs to both the framed mapping class group and the Torelli group, acts on $Z := Z^1_\chi(\bR)$, where $\chi \in H^1(\Sigma_g, \bR_{>0})$ by a dense subgroup of $G = \mathrm{SL}(Z)$. Imitating the strategy of producing an invariant measure in the scaled case would therefore produce a $G$-invariant measure on 
\[ \{ [v_1, v_2] \in (Z\times Z)/\bR_{>0} : \text{$v_1$ and $v_2$ are not collinear} \}  \]
which is a homogeneous space for the action of $G$, i.e. the space can be written as $G/H$ where $H$ is a subgroup of $G$. But $G/H$ admits a $G$-invariant measure if and only if $H$ acts on the Lie algebra of $G$ by matrices of determinant one. One may check that, for genus greater than one, this does not hold for every matrix in $H$. 
\end{rem}

\appendix
   
\section{Relative framings and the twisted cohomology of the mapping class group}\label{appendix:A} 
The purpose of this section is to establish the cohomological results used in the body of the paper. The main results are Theorem \ref{T:MainMCGCohomology} stated just below, and Proposition \ref{P:ArcFramedMCG}, recalled in Section \ref{SS:Aarcframed}.

\para{Convention} Throughout this section, $A$ denotes either $\bZ$ or else a field of characteristic zero. If left unspecified, coefficients will be taken as such $A$.

\begin{thm}\label{T:MainMCGCohomology}
For $g \geq 2$ and $n \ne 1$,   $H^1(\Mod_{g,n+1}, H_1(\Sigma_{g,n};A)) \cong A$.
Moreover, the change of winding number functional $\overline{C(\xi_0)}$ described in Lemma \ref{L:barCpsi} is a well-defined cocycle that generates $H^1(\Mod_{g,{n+1}}, H_1(\Sigma_{g,n};A))$.
\end{thm}

\begin{rem}
    Here is a brief description of the change of winding number functional $\overline{C(\xi_0)}$ constructed in Lemma \ref{L:barCpsi}. Let $\xi_0$ be a non-vanishing vector field on $\Sigma_{g,n+1}$ having zero winding number around the set $P$ of the first $n$ marked points. Following Proposition \ref{prop:ewnf}, $\xi_0$ defines a relative winding number function $W_{\xi_0}^+$ defined on the set of simple closed curves on $\Sigma_{g,n+1}$ along with the set of arcs on $\Sigma_{g,n+1}$ with both endpoints at $P$. Given a mapping class $f \in \Mod_{g,n}$, it is shown in Lemma \ref{L:barCpsi} that the change in winding number $W_{\xi_0}^+((f_*-1)\cdot)$ is a well-defined functional on $H_1(\Sigma_{g},P)$, which by Lefschetz duality corresponds to an element of the coefficient module $H_1(\Sigma_{g,n}) = H_1(\Sigma_g \setminus P)$. 
\end{rem}

\begin{rem}
    The case $n = 0$ is a classical result. The calculation $H^1(\Mod_{g,1}, H_1(\Sigma_{g};A)) \cong A$ was obtained by Morita (in the case $A = \bZ$) in \cite{Morita-FamiliesOfJacobians}, and the description of a generating cocycle appeared in the work of Furuta (as recorded in \cite[p. 569]{morita-secondary}); see also the work of Trapp \cite{trapp}. In the body of the appendix we will assume $n \ge 2$ throughout.
\end{rem}

While our main results are formulated for ease of use in the main body of the text, the proof strategy takes a somewhat different path. Namely, it is technically easier to work in the setting of surfaces with boundary components, which we will do through the first two subsections. We notate boundary components with superscripts, so that $\Mod_g^b$ denotes the mapping class group of the surface $\Sigma_g^b$ with $b$ boundary components; by definition such mapping classes fix the boundary pointwise. 

Our first objective (Theorem \ref{T:MCGbdrycohom}) is to understand cohomology in the setting of boundary components, with coefficients in the module 
\[
H_g^b:= H_1(\Sigma_g^b).
\]
For later use, it is important to have an explicit basis for this space. This will arise via a connecting map $C: H^1(U_g^b, \tilde P) \to Z^1(\Mod_g^b; H_g^b)$; here $U_g^b$ is the unit tangent bundle of the surface $\Sigma_g^b$, and $\tilde P$ is a collection of $b$ points, one on each boundary component (this is discussed further in Section \ref{SSS:lower}).

\begin{thm}\label{T:MCGbdrycohom}
    For $g \ge 3$ and $b\ge 3$, there is an isomorphism 
    \[
    H^1(\Mod_g^{b}; H_g^b) \cong A^{b}.
    \]
    At the cocycle level, this is realized by an isomorphism
    \[
    C: H^1(U_g^b, \tilde P) \to Z^1(\Mod_g^b; H_g^b).
    \]
\end{thm}

Once this has been established, Theorem \ref{T:MainMCGCohomology} is proved in Section \ref{SS:Aproofs}, and Proposition \ref{P:ArcFramedMCG} is established in Section \ref{SS:Aarcframed}.

    \subsection{Computation in the boundary component setting}\label{SS:A2}

Our objective here is Theorem \ref{T:MCGbdrycohom}: we seek to show that $H^1(\Mod_g^b; H_g^b) \cong A^b$, and that moreover this isomorphism is realized at the cocycle level by $C$. We first establish Proposition \ref{P:upperbound}, showing that the rank of $H^1(\Mod_g^b; H_g^b)$ is at most $b$. To prove Theorem \ref{T:MCGbdrycohom}, we then establish injectivity of $C$.

        \subsubsection{An upper bound on the rank} \label{SSS:upper}
In this section, we prove the following assertion.

\begin{prop}
    \label{P:upperbound}
    For $g \ge 2$ and $b \ge 2$, $H^1(\Mod_g^b; H_g^b)$ is a free $A$-module of rank at most $b$.
\end{prop}

This will follow from an analysis of a Birman exact sequence. Consider the capping map $\Sigma_g^b \to \Sigma_g^{b-1}$ capping $\Delta_b$ off with a disk. This produces two short exact sequences of relevance. The first is a short exact sequence of $\Mod_g^b$-modules
\[
0 \to A \to H_g^b \to H_g^{b-1} \to 0,
\]
with $A$ generated by $[\Delta_b]$. The second (the Birman exact sequence) is the short exact sequence of groups
\[
1 \to U_g^{b-1} \to \Mod_g^b \to \Mod_g^{b-1} \to 1,
\]
where $U_g^{b-1}$ denotes the fundamental group of the unit tangent bundle of the capped-off surface $\Sigma_g^{b-1}$.

\begin{lem}\label{P:ResToU}
For all $g \ge 2, b \ge 2$, the restriction homomorphism from $H^1(\Mod_{g}^{b}; H_g^b)$ to $H^1(U_g^{b-1}; H_g^b)^{\Mod_{g}^{b}}$ is an isomorphism. 
\end{lem}
\begin{proof}
We will examine the inflation-restriction sequence associated to the Birman exact sequence
    \begin{equation}\label{eqn:ses}
        1 \to U_g^{b-1} \to \Mod_{g}^{b} \to \Mod_{g}^{b-1} \to 1.
    \end{equation}
    A first observation is that since $b \ge 2$, this sequence splits. Indeed, consider
    \begin{equation*}\label{eqn:splitting}
        \Sigma_{g}^{b-1} \to \Sigma_{g}^{b} \to \Sigma_{g}^{b-1},
    \end{equation*}
    where the first map attaches one boundary component of a pair of pants to $\Delta_{b-1}$, and the second caps $\Delta_b$ with a disk. On the level of mapping class groups, this induces the claimed splitting. 

    The inflation-restriction sequence for a split short exact sequence degenerates to a short exact sequence. In this setting of \eqref{eqn:ses} with coefficients in $H_g^b$, this produces
    \[
    0 \to H^1(\Mod_g^{b-1}; (H_g^b)^{U_g^{b-1}}) \to H^1(\Mod_g^b; H_g^b) \to H^1(U_g^{b-1}; H_g^b)^{\Mod_g^b} \to 0.
    \]
    Thus the claim of Lemma \ref{P:ResToU} is equivalent to the assertion $H^1(\Mod_g^{b-1}; (H_g^b)^{U_g^{b-1}})=0$. To see this, we first investigate the structure of $H_g^b$ as a $U_g^{b-1}$-module.

    \begin{lem}\label{lemma:pushformula}
        Let $\gamma \in U_g^{b-1}$ and $x \in H_g^b$ be given. Then
        \[
        \gamma \cdot x = x + \pair{\bar\gamma,x}\Delta_b,
        \]
        where $\bar\gamma \in H_{g}^{b}$ is the image of $\gamma$ in homology under the natural projection $H_1(U_g^{b-1}) \to H_1(\Sigma_{g}^{b-1})$, lifted arbitrarily to $H_g^b$ (well-defined since $\Delta_b$ is in the kernel of $\langle \cdot, \cdot \rangle$). 
    \end{lem}
    \begin{proof}
        $\gamma \in U_g^b$ acts on $\Sigma_{g}^{b}$ by pushing $\Delta_b$ around on the framed path specified by $\gamma$. Each crossing of $\gamma$ with (a representative of) $x$ alters $x$ by adding or subtracting $\Delta_b$ (according to the sign of the crossing). 
    \end{proof}

    Define $K_g^b \le H_g^b$ as the subspace spanned by the boundary components. It enjoys the following properties (the proofs of which are easy exercises).

    \begin{lem}\label{L:Kprops}
        Let $K_g^b \le H_g^b$ be the subspace spanned by the classes of the boundary components $\Delta_1, \dots, \Delta_b$. Then the following properties hold.
        \begin{enumerate}
            \item $K_g^b$ is free of rank $b-1$, with the relation $\Delta_1 + \dots + \Delta_b = 0$ holding.
            \item $K_g^b$  is the kernel of the map $H_g^b \to H_1(\Sigma_g)$ induced by capping each boundary component, and is also the kernel of the map $H_g^b \to H_1(\Sigma_g^b, \partial)$ induced by the long exact sequence of the pair $(\Sigma_g^b, \partial)$, where $\partial$ denotes the boundary of $\Sigma_g^b$.
            \item $(H_g^b)^{\Mod_g^b} = K_g^b$.
            \item $K_g^b$ is the kernel of the intersection form $\pair{\cdot, \cdot}: H_g^b \times H_g^b \to A$.
        \end{enumerate}
    \end{lem}
    
    \begin{cor}\label{cor:huk}
        $(H_g^b)^{U_g^{b-1}} = K_g^b$.
    \end{cor}
    \begin{proof}
        By the formula of Lemma \ref{lemma:pushformula}, $x \in H_g^b$ is globally invariant if and only if it lies in the kernel of the intersection form.
    \end{proof}
    \begin{cor}
        $H^1(\Mod_{g}^{b-1}; (H_g^b)^{U_g^{b-1}}) = 0$, for any $g \ge 2$.
    \end{cor}
    \begin{proof}
        By Corollary \ref{cor:huk},
        \[
        H^1(\Mod_{g}^{b-1}; (H_g^b)^{U_g^{b-1}}) = H^1(\Mod_{g}^{b-1}; K_g^b).
        \]
        Now $K_g^b$ is torsion-free, and trivial as a $\Mod_{g}^{b-1}$-module. Since $H_1(\Mod_g^{b-1})$ is finite cyclic for $g =2$ and trivial for $g \ge 3$, it follows that $H^1(\Mod_{g}^{b-1};K_g^b) = 0$ as claimed.
    \end{proof}
    This completes the proof of Lemma \ref{P:ResToU}.
\end{proof}

Following Lemma \ref{P:ResToU}, our attention shifts to $H^1(U_g^{b-1};H_g^b)^{\Mod_g^b}$. This can be studied by means of the long exact sequence in cohomology induced by the short exact sequence
\[
0 \to A \to H_g^b \to H_g^{b-1} \to 0.
\]
Identifying $(H_g^b)^{U_g^{b-1}} = K_g^b$, the long exact sequence becomes
\begin{equation} \label{eqn:ULES} 
    0 \to A \to K_g^b \to H_g^{b-1}
         \to H^1(U_g^{b-1}; A) \to H^1(U_g^{b-1};H_g^b) \to H^1(U_g^{b-1}; H_g^{b-1}).
\end{equation}

To study $H^1(U_g^{b-1};H_g^b)$ (and ultimately $H^1(U_g^{b-1};H_g^b)^{\Mod_g^b}$), we are interested in the $\Mod_g^b$-modules $H^1(U_g^{b-1}; A)$ and $H^1(U_g^{b-1}; H_g^{b-1})$. In Corollary \ref{C:UtoHInvariants} below, we compute $H^1(U_g^{b-1}; H_g^{b-1})$; this will follow from the next sequence of results.

\begin{lem}\label{L:LinAlgLemma}
Suppose that $\Gamma$ is a group acting linearly on a vector space $V$. Suppose that $W$ is the subspace of $\Gamma$-fixed points. Suppose too that the only $\Gamma$-equivariant maps from $V/W$ to itself are constant multiples of the identity. Suppose finally that there is a vector $v_0 \in V$ so that $\Gamma \cdot v_0$ is not contained in any nontrivial proper affine subspace. Then the only equivariant linear maps from $V$ to $V$ are constant multiples of the identity.
\end{lem}
The statement and proof work equally well when replacing the category of vector spaces with the category of $\bZ$-modules.
\begin{proof}
Let $f: V \ra V$ be any equivariant linear map. It must take $W$ to itself and hence it induces an equivariant map from $V/W$ to itself, which is necessarily constant. Therefore, there is a constant $c$ so that $f(v) - cv$ is an equivariant linear map from $V$ to itself whose image belongs to $W$, i.e. this map is $\Gamma$-invariant. The preimage of any point is an affine subspace. However, $\Gamma \cdot v_0$ is not contained in any nontrivial proper affine subspace and so the map must be the zero map. 
\end{proof}

\begin{cor}\label{C:EquivariantH}
There is an isomorphism $\mathrm{Hom}(H_g^{b-1}, H_g^{b-1})^{\Mod_{g}^b} \cong A$, with the left-hand side spanned by multiples of the identity.
\end{cor}
\begin{proof}
Set $V = H_g^{b-1}$ and $W = K_g^{b-1}$. Then $V/W \cong H_1(\Sigma_g)$, which is an irreducible $\Mod_{g}^b$-module and so the only equivariant maps from $V/W$ to itself are constant multiples of the identity by Schur's lemma. Notice that if $v_0 \in H_g^{b-1}$ projects to a primitive vector in $H_1(\Sigma_g)$, then its mapping class group orbit is a $K$-torsor (by Lemma \ref{lemma:pushformula}) whose image in $H_1(\Sigma_g)$ is the set of primitive vectors. This set is not contained in any affine subspace of $H_g^{b-1}$. The claim now follows from the previous lemma. 
\end{proof}

\begin{cor}\label{C:UtoHInvariants}
$\mathrm{Hom}(U_g^{b-1}, H_g^{b-1})^{\Mod_{g}^b} \cong A$. In particular, every element is a multiple of the projection map from $U_g^{b-1}$ to $H_g^{b-1}$.
\end{cor}
\begin{proof}
We apply the inflation-restriction sequence associated to the split exact sequence
\[
0 \to \bZ \to U_g^{b-1} \to \pi_1(\Sigma_g^{b-1}) \to 1
\]
induced by the bundle projection, taking coefficients in $H_g^{b-1}$; note that by Lemma \ref{lemma:pushformula}, this is a trivial $U_g^{b-1}$-module. As before, since the sequence of groups is split, the inflation-restriction sequence degenerates to a short exact sequence:
\[
0 \to H^1(\Sigma_g^{b-1}; H_g^{b-1}) \to H^1(U_g^{b-1};H_g^{b-1}) \to H^1(\bZ; H_g^{b-1}) \to 0.
\]
Applying the universal coefficients theorem, this can be understood as
\[
0 \to \Hom(H_g^{b-1}, H_g^{b-1}) \to \Hom(U_g^{b-1},H_g^{b-1}) \to H_g^{b-1} \to 0.
\]
This is a short exact sequence of $\Mod_g^b$-modules. We claim that 
\[
(\Hom(U_g^{b-1},H_g^{b-1}))^{\Mod_g^b} \le \Hom(H_g^{b-1}, H_g^{b-1}).
\]
As $(\Hom(H_g^{b-1}, H_g^{b-1}))^{\Mod_g^b} \cong A$ by the previous lemma, this will establish the claim.

The claim $(\Hom(U_g^{b-1},H_g^{b-1}))^{\Mod_g^b} \le \Hom(H_g^{b-1}, H_g^{b-1})$ is equivalent to the assertion that $f(\zeta) = 0$ for any $f \in (\Hom(U_g^{b-1},H_g^{b-1}))^{\Mod_g^b}$, where as usual $\zeta$ denotes the class of the fiber. A simple preliminary observation is that the action of $\Mod_g^b$ on this space factors through $\Mod_g^{b-1}$. To see that $f(\zeta) = 0$ for any $\Mod_g^{b-1}$-equivariant $f$, let $c, c'$ be a pair of nonseparating simple closed curves on $\Sigma_g^{b-1}$ cobounding a subsurface $\Sigma_1^2\le \Sigma_g^{b-1}$, and let $a$ be a simple closed curve on $\Sigma_g^{b-1}$ satisfying $i(a,c) = i(a, c') = 1$. Let $\phi = T_c T_{c'}^{-1} \in \Mod_g^{b-1}$ be the bounding pair map, and set $a' = \phi(a)$. Then $a, a'$ cobound a subsurface $\Sigma_1^2 \subset \Sigma_g^{b-1}$. Lifting $a, a'$ to homology classes on $U_g^{b-1}$ via the Johnson lift, homological coherence (equation \eqref{eqn:homcoh}) implies that $[a'] = [a] + 2 \zeta$ when suitably oriented. 

This construction enforces the condition $f(\zeta) = 0$. From the computation $[a'] = [a] + 2 \zeta$,
\[
f(a') = f(a + 2 \zeta) = f(a) + 2 f(\zeta).
\]
But by equivariance,
\[
f(a') = f(\phi a) = \phi f(a) = f(a),
\]
since $\phi = T_c T_{c'}^{-1}$ acts {\em trivially} on $H_g^{b-1}$.

\end{proof}

\begin{proof}[Proof of Proposition \ref{P:upperbound}]
   By Lemma \ref{P:ResToU}, it suffices to show that $H^1(U_g^{b-1};H_g^b)^{\Mod_g^b}$ is free of rank at most $b$. We do so by analyzing the exact sequence
   \begin{align*}
        0 \to A \to K_g^{b} \to H_g^{b-1}
         \to H^1(U_g^{b-1}; A) \to H^1(U_g^{b-1};H_g^b) \to H^1(U_g^{b-1}; H_g^{b-1}).
    \end{align*}
    of \eqref{eqn:ULES}. A first observation is that since $H^1(U_g^{b-1};A)$ and $H^1(U_g^{b-1}, H_g^{b-1})$ are free (as Abelian groups), so too must be the term $H^1(U_g^{b-1}; H_g^b)$ in between. The first four terms in the sequence are free of ranks $1, b-1, 2g+b-2,$ and $2g+b-1$, respectively; exactness then implies that the image of $H^1(U_g^{b-1};A)$ in $H^1(U_g^{b-1};H_g^b)$ has rank $b-1$. A closer look shows that the image of $H^1(U_g^{b-1};A)$ in $H^1(U_g^{b-1};H_g^b)$ can be identified with $K_g^{b-1}$ - this identification is at the level of $\Mod_g^b$-modules, and so in particular, $K_g^{b-1} \le H^1(U_g^{b-1};H_g^b)^{\Mod_g^b}$.

    By Corollary \ref{C:UtoHInvariants}, $H^1(U_g^{b-1};H_g^{b-1})^{\Mod_g^b}$ has rank $1$. Since any element of $H^1(U_g^{b-1}; H_g^b)^{\Mod_g^b}$ projects to $H^1(U_g^{b-1};H_g^{b-1})^{\Mod_g^b}$ under the change-of-coefficients map, it follows that $H^1(U_g^{b-1}; H_g^b)^{\Mod_g^b}$ has rank at most $(b-1) +1 = b$ as claimed.
\end{proof}

        \subsubsection{The lower bound} \label{SSS:lower}

The cocycles spanning $H^1(\Mod_g^b; H_g^b)$ will be constructed from the connecting map associated to a short exact sequence of coefficients. Our first task is to describe this sequence of coefficients, which arises from the unit tangent bundle of the surface.

The boundary of $\Sigma_g^b$ will be denoted $\partial$. Enumerate the components as
\[
\partial = \Delta_1 \sqcup \dots \sqcup \Delta_b,
\]
and let $P = \{p_1, \dots, p_b\}$ be a collection of points with $p_i \in \Delta_i$. Let $U_g^b$ denote the unit tangent bundle of $\Sigma_g^b$, and let $\tilde \partial$ denote its boundary. Then $\tilde \partial$ is a union of $b$ tori $\tilde \Delta_i$. We fix a basis for $H_1(\tilde \Delta_i)$: let $\ell_i$ be the $S^1$ fiber over $p_i \in \Delta_i$, and let $m_i$ be some meridian, projecting down to $[\Delta_i] \in H_1(\Sigma_g^b)$. Finally, let $\tilde P$ be a set of lifts of $P$ to $U_g^b$. 

\begin{lem}
    For $g \ge 2$ and $b \ge 1$, there is a short exact sequence of $\Mod_g^b$-modules
    \[
    0 \to H_g^b \to H^1(U_g^b, \tilde P) \to A^b \to 0,
    \]
    with $A^b$ carrying the trivial action. $A^b$ can be identified with the subspace of $H_1(\tilde \partial)$ spanned by $m_1 + \dots + m_b, \ell_1 - \ell_b, \dots, \ell_{b-1} - \ell_b$.
\end{lem}

\begin{proof}
    This will arise from the long exact sequence in cohomology for the triple $(U_g^b, \tilde \partial, \tilde P)$. This begins with
    \[
    0 \to H^1(U_g^b, \tilde \partial) \to H^1(U_g^b, \tilde P) \to H^1(\tilde \partial, \tilde P) \xrightarrow{\delta} H^2(U_g^b, \tilde \partial) \to \dots
    \]
    We claim that $H^1(U_g^b, \tilde \partial) \cong H_1(\Sigma_g^b) = H_g^b$. Indeed, by Lefschetz duality, $H^1(U_g^b, \tilde \partial) \cong H_2(U_g^b)$. We claim that $H_2(U_g^b) \cong H_1(\Sigma_g^b)$. Since $b \ge 1$, $U_g^b \to \Sigma_g^b$ admits a section (given by some nowhere-vanishing vector field) and hence is a product: $U_g^b \cong \Sigma_g^b \times S^1$. By the K\"unneth formula, since $H_2(\Sigma_g^b) = 0$, the claim follows. 

    Making the identification $H^1(U_g^b; \tilde \partial) \cong H_g^b$, we thus obtain the following short exact sequence:
    \[
    0 \to H_g^b \to H^1(U_g^b; \tilde P) \to \ker \delta \to 0.
    \]
    Our remaining task is to identify $\ker \delta$ with the span of $m_1 + \dots + m_b, \ell_1 - \ell_b, \dots, \ell_{b-1} - \ell_b$ inside $H_1(\tilde \partial)$. To begin with, $H^1(\tilde \partial, \tilde P) \cong H^1(\tilde \partial) \cong H_1(\tilde \partial)$ by Poincar\'e duality. Again by Lefschetz duality, $H^2(U_g^b; \tilde \partial) \cong H_1(U_g^b)$, and so $\delta: H^1(\tilde \partial, \tilde P) \to H^2(U_g^b, \tilde \partial)$ can be identified with the inclusion map $i: H_1(\tilde \partial) \to H_1(U_g^b) \cong H_g^b \oplus H_1(S^1)$. Under this identification, $i(m_i) = [\Delta_i]$, and $i(\ell_i) = [\zeta]$, the class of the $S^1$ fiber. Since $\sum [\Delta_i] = 0$ in $H_g^b$, the claim follows. 
\end{proof}

The long exact sequence in cohomology associated to this short exact sequence of coefficients begins
\[
0 \to (H_g^b)^{\Mod_g^b} \to (H^1(U_g^b, \tilde P))^{\Mod_g^b} \to \ker \delta \to H^1(\Mod_g^b; H_g^b).
\]
\begin{lem}\label{L:lowerbound}
    For $g \ge 2, b \ge 1$, the map $(H_g^b)^{\Mod_g^b} \to (H^1(U_g^b, \tilde P))^{\Mod_g^b}$ is an isomorphism. Consequently, $\ker \delta \to H^1(\Mod_g^b; H_g^b)$ is an injection.
\end{lem}
\begin{proof}
    We will show that each of $(H_g^b)^{\Mod_g^b}$ and $(H^1(U_g^b, \tilde P))^{\Mod_g^b}$ are free of rank $b-1$. Since $\ker \delta \le H^1(\tilde \partial, \tilde P)$ is also free, this will show the claim. 

    Recall from Lemma \ref{L:Kprops} that $(H_g^b)^{\Mod_g^b} = K_g^b$ is free of rank $b-1$. It remains to show that $(H^1(U_g^b, \tilde P))^{\Mod_g^b} \cong A^{b-1}$. The long exact sequence for the pair $(U, \tilde P)$ induces the short exact sequence
    \[
    0 \to \tilde H^0(\tilde P) \to H^1(U_g^b, \tilde P) \to H^1(U_g^b) \to 0.
    \]
    note that $\tilde H^0(\tilde P)$ is a trivial $\Mod_g^b$-submodule of rank $b-1$. Thus to finish the argument, it suffices to establish the following.\\
    
    \para{Claim} $H^1(U_g^b)^{\Mod_g^b} = 0$; consequently $(H^1(U_g^b, \tilde P))^{\Mod_g^b} = \tilde H^0(\tilde P)$. \\

    To establish the claim, we examine the sequence
    \[
    0 \to H^1(\Sigma_g^b) \to H^1(U_g^b) \to H^1(S^1) \to 0.
    \]
    A first observation is that $H^1(\Sigma_g^b)^{\Mod_g^b} = 0$. For this, we identify $H^1(\Sigma_g^b) \cong H_1(\Sigma_g^b, \partial)$. By Lefschetz duality, this space is dual under the intersection pairing to $H_1(\Sigma_g^b)$. Thus every $x \in H_1(\Sigma_g^b, \partial)$ is moved by the action of some Dehn twist $T_c$ such that $\pair{[c],x} \ne 0$. 

    We finally claim that $H^1(U_g^b)^{\Mod_g^b} \le H^1(\Sigma_g^b)$. Let $\psi \in H^1(U_g^b)$ be a class with nontrivial restriction to $H^1(S^1)$. Let $c \subset \Sigma_g^b$ be a nonseparating simple closed curve, let $c'$ be disjoint and such that $c \cup c'$ bounds a subsurface $\Sigma_1^2 \subset \Sigma_g^b$, and let $f \in \Mod_g^b$ satisfy $f(c) = c'$. Let $\hat c,\hat c'$ be the Johnson lifts of $c, c'$ (qv. Section \ref{subsection:wnf}). Then by homological coherence (Equation \eqref{eqn:homcoh}), suitably oriented, $[c'] = [c] + 2[\zeta]$. Consequently, 
    \[
    f^*\psi(c') = \psi(c) = \psi(c') - 2 \psi(\zeta) \ne \psi(c),
    \]
    showing that any such $\psi$ is not globally fixed.
\end{proof}

\begin{proof}[Proof of Theorem \ref{T:MCGbdrycohom}]
Proposition \ref{P:upperbound} shows that $H^1(\Mod_g^b; H_g^b)$ is free of rank at most $b$, while Lemma \ref{L:lowerbound} shows that the rank is at least $b$. Thus the connecting map $\ker \delta \to H^1(\Mod_g^b; H_g^b)$ is an {\em isomorphism}. In general, if $0 \to U \to V \to W \to 0$ is a short exact sequence of $G$-modules with $W^G = W$ and the connecting map $\bar C: W^G \to H^1(G;U)$ an isomorphism, this lifts to the cocycle level as an isomorphism $C: V \to Z^1(G; U)$, with $C(v)$ given by the formula $C(v)(g) = (g-1)v \in U$. 
\end{proof}

        \subsubsection{Construction of cocycles} \label{SSS:cocycles}
            The purpose of this section is to give explicit descriptions of the isomorphisms $\bar C: \ker \delta \to H^1(\Mod_g^b; H_g^b)$ and $C: H^1(U_g^b, \tilde P) \to Z^1(\Mod_g^b; H_g^b)$. Recall that $\ker \delta \le H_1(\tilde \partial)$ has a basis given by $\ell_1- \ell_b, \dots, \ell_{b-1} - \ell_b, \sum m_i$, where $\ell_i$ is a longitude (the $S^1$ fiber above a fixed point in $\Delta_i$), and $m_i$ is a meridian, projecting down to $\Delta_i$. 
            After piecing the various identifications together, the classes $\ell_i - \ell_b$ admit a very simple description, as a ``change of relative homology class'' functional.

            \begin{construction}
            \label{L:lij}
                For $1 \le i < j \le b$, the image $\bar C(\ell_i - \ell_j) \in H^1(\Mod_g^b; H_g^b)$ can be described at the cocycle level as follows: choose an arc $\alpha_{ij} \subset \Sigma_g^b$ connecting $\Delta_i$ to $\Delta_j$. Given $f \in \Mod_g^b$, define
                \[
                C(\alpha_{ij})(f) = f_*\alpha_{ij} - \alpha_{ij} \in H_g^b.
                \]
                \qed
            \end{construction}

           The description of $\bar C(\sum m_i)$ is slightly more complicated, involving the theory of relative winding number functions as developed in Section \ref{SS:relwnf}. Recall the setup: one chooses a set $\wt P$ of basepoints, one on each component of $\partial U_g^b$. As explained in Lemma \ref{L:relwnf}, a non-vanishing vector field $\psi$ avoiding $\tilde P$ determines, by Lefschetz duality, a class $W_\psi \in H^1(U_g^b,\wt P)$, a relative winding number function.

           \begin{construction}\label{L:summi}
               Let $\psi: \Sigma_g^b \to U_g^b$ be a non-vanishing vector field avoiding $\tilde P$. This defines a map
               \[
               C(\psi): \Mod_g^b \to (H_1(\Sigma_g^b, \partial))^* \cong H_g^b
               \]
               as follows: given $f \in \Mod_g^b$ and $a \in H_1(\Sigma_g^b, \partial)$, realize $a$ as a (sum of) arcs and simple closed curves in $U_g^b$ based at $\tilde P$; call this $\alpha$. Then set 
               \[
               C(\psi)(f)(a) = W_\psi^+(f(\alpha)) - W_\psi^+(\alpha) = W_\psi^+((f_*-1)\alpha). 
               \]\qed
           \end{construction}

   Let $\psi: \Sigma_g^b \to U_g^b$ be a non-vanishing vector field. The {\em boundary signature} of $\psi$ is the tuple $B_\psi = (W_\psi(\Delta_1), \dots, W_\psi(\Delta_b))$.

        \begin{lem}\label{L:cohomVF}
            Let $\psi, \phi: \Sigma_g^b \to U_g^b$ be non-vanishing vector fields avoiding $\tilde P$. Then there is an equality $[C(\psi)] = [C(\phi)]$ in $H^1(\Mod_g^b; H_g^b)$ if and only if $B_\psi = B_\phi$. Moreover, if $C \in Z^1(\Mod_g^b, H_g^b)$ is cohomologous to some $C(\psi)$, then $C = C(\phi)$ for some non-vanishing vector field $\phi: \Sigma_g^b \to U_g^b$ avoiding $\tilde P$, necessarily with $B_\psi = B_\phi$.
        \end{lem}

        \begin{proof}
            By definition, the classes $C(\psi), C(\phi)$ are cohomologous if and only if there is $y \in H_g^b$ such that for all $a \in H_1(\Sigma_g^b, \partial)$ and all $f \in \Mod_g^b$,
            \begin{align}\label{E:Wcoboundary}
                (W_\psi^+ - W_\phi^+)((f_*-1)a) = \pair{(f_*-1)y,a}.
            \end{align}
            As in Proposition \ref{prop:framingchar}, the difference $W_\psi^+ - W_\phi^+ \in H^1(U_g^b,\tilde P)$ can be expressed as 
            \[
            W_\psi^+ - W_\phi^+ = p^*v
            \]
            for some $v \in H^1(\Sigma_g^b, P)$, where $p: U_g^b \to \Sigma_g^b$ is the projection and $P = p(\tilde P) \subset \Sigma_g^b$.

            It is readily seen that the condition $B_\psi = B_\phi$ is equivalent to the condition $v(\Delta_i) = 0$ for all components $\Delta_i \in \partial \Sigma_g^b$, and that in turn this is equivalent to the condition $v \in H^1(\Sigma_g,P) \le H^1(\Sigma_g^b,P) \cong H_g^b$. Thus if $B_\psi = B_\phi$, then after making the necessary identifications, $y = v$ is the required element. 
            
            Conversely, if $[C(\psi)] = [C(\phi)]$ and \eqref{E:Wcoboundary} holds, one can choose an arc $a_i$ connecting $\Delta_i$ to some other boundary component. Then $f = T_{\Delta_i}$ satisfies $(f_*-1)a_i = \Delta_i$, so the left-hand side of \eqref{E:Wcoboundary} measures the difference $W^+_\psi(\Delta_i) - W^+_\phi(\Delta_i)$. On the right hand side, 
            \[
            \pair{(f_*-1)y, a_i} = \pair{y, (f_*^{-1}-1)a_i} = \pair{y, - \Delta_i} = 0,
            \]
            since $[\Delta_i] = 0$ in $H_1(\Sigma_g^b, \partial)$. This proves the first claim.

            To prove the second claim, suppose that $C \in Z^1(\Mod_g^b, H_g^b)$ is cohomologous to some $C(\psi)$: say $C = C(\psi) + \delta(y)$ for $y \in H_g^b$ and $\delta: H_g^b \to Z^1(\Mod_g^b, H_g^b)$ the differential. Arguing as above, $y \in H_g^b \cong H^1(\Sigma_g,P)$ pulls back to $p^*y \in H^1(U_g^b, \tilde P)$, and then (following Proposition \ref{prop:framingchar}) defining $\phi$ to be a non-vanishing vector field in the cohomology class $W^+_\phi = W^+_\psi + p^*y \in H^1(U_g^b, \tilde P)$, the claim follows. 
        \end{proof}

    \subsection{Descending to the punctured setting}\label{SS:Aproofs}

        \subsubsection{From boundary components to punctures}
            \begin{prop} \label{P:puncturevanish}
                For $g \ge 2$ and $b \ge 2$,
                \[
                H^1(\Mod_{g,b}; H_g^b) = 0.
                \]
            \end{prop}

            \begin{proof}
                We analyze the inflation-restriction sequence associated to the short exact sequence of groups
                \[
                0 \to \bZ^b \to \Mod_g^b \to \Mod_{g,b} \to 1
                \]
                induced by capping all boundary components with punctured disks. Taking coefficients in $H_g^b$ (recalling that $H_g^b = H_1(\Sigma_g^b) \cong H_1(\Sigma_{g,b})$, so that $H_g^b$ is a $\Mod_{g,b}$-module in a natural way), this begins
                \[
                0 \to H^1(\Mod_{g,b};H_g^b) \to H^1(\Mod_g^b; H_g^b) \to H^1(\bZ^b; H_g^b)^{\Mod_g^b}.
                \]
                To prove the claim, it therefore suffices to show that $H^1(\Mod_g^b; H_g^b) \to H^1(\bZ^b; H_g^b)^{\Mod_g^b}$ is an injection.

                To begin with, the target $H^1(\bZ^b; H_g^b)^{\Mod_g^b}$ can be simplified to $\Hom(\bZ^b,K_g^b)$. The work of Section \ref{SSS:cocycles} provides us with an explicit set of cocycle representatives forming a basis for $H^1(\Mod_g^b; H_g^b)$. As in  Construction \ref{L:lij}, we let $\alpha_{ib}$ be an arc connecting $\Delta_i, \Delta_b$, and then the restriction of $C(\alpha_{ib})$ to $\bZ^b$  is computed as
                \[
                C(\alpha_{ib})(\Delta_j) = \begin{cases}
                                                \Delta_i & j = i\\
                                                -\Delta_b & j = b\\
                                                0           &\mbox{otherwise.}
                        \end{cases}
                \]
                To compute the image of $\bar C(\sum m_i)$, let $\psi$ be a non-vanishing vector field on $\Sigma_g^b$. By the twist-linearity formula, the effect of $T_{\Delta_i}$ on an arc based at $\Delta_i$ is to add $W_\psi(\Delta_i)$. After making the necessary identifications, this shows that
                \begin{equation}\label{E:cpsi}
                    C(\psi)(\Delta_j) = W_\psi(\Delta_j) \Delta_j.
                \end{equation}
             
                These functions are linearly independent in $\Hom(\bZ^b, K_g^b)$. To see this, recall that $K_g^b = \pair{\Delta_1, \dots, \Delta_b}/\sum \Delta_i$; we therefore identify $K_g^b$ with the subspace spanned by $\Delta_1, \dots, \Delta_{b-1}$ by systematically identifying $\Delta_b = -(\Delta_1 + \dots + \Delta_{b-1})$ (though this will not be relevant in what follows).
                
                Linear independence of the set spanned by the classes $C(\alpha_{ib})$ is clear, since only $C(\alpha_{ib})$ takes a nonzero value on $\Delta_i$ for $1 \le i \le b-1$. For this same reason, if $C(\psi)$ could be expressed in the span of such classes, then necessarily such an expression would have to be
                \[
                C(\psi) = \sum_{i = 1}^{b-1} W_\psi(\Delta_i) C(\alpha_{ib}),
                \]
                in order to make the values on $\Delta_1, \dots, \Delta_{b-1}$ correct. But by homological coherence (Equation \eqref{eqn:homcoh}),
                \[
                C(\psi)(\Delta_b) = W_\psi(\Delta_b)\Delta_b = \left(\chi(\Sigma_g^b) - \sum_{i = 1}^{b-1} W_\psi(\Delta_i)\right) \Delta_b,
                \]
                while
                \[
                \sum_{i = 1}^{b-1} W_\psi(\Delta_i) C(\alpha_{ib}) (\Delta_b) = \left(- \sum_{i = 1}^{b-1} W_\psi(\Delta_i)\right) \Delta_b.
                \]
                As $\chi(\Sigma_g^b) \ne 0$, this proves linear independence.
            \end{proof}

        \subsubsection{Forgetting a puncture}

        \begin{lem}\label{L:barCpsi}
            Let $\xi_0$ be a non-vanishing vector field on $\Sigma_g^b$ for which $W_{\xi_0}(\Delta_1) = \dots = W_{\xi_0}(\Delta_{b-1}) = 0$. Then $C(\xi_0) \in Z^1(\Mod_g^b; H_g^b)$ descends to a cocycle $\overline{C(\xi_0)} \in Z^1(\Mod_{g,b}; H_g^{b-1})$.
        \end{lem}
        \begin{proof}
            From \eqref{E:cpsi}, $C(\xi_0)$ vanishes on $\Delta_1, \dots, \Delta_{b-1} \le \bZ^b \le \Mod_g^b$, and $C(\xi_0)(\Delta_b) = \chi(\Sigma_g^b)\Delta_b$. Modding out by $\Delta_b \le H_g^b$, this shows that the projection $\overline{C(\xi_0)}: \Mod_g^b \to H_g^{b-1}$ vanishes on the kernel $\bZ^b$ of the boundary-capping map $\Mod_g^b \to \Mod_{g,b}$, and hence descends to $\overline{C(\xi_0)} \in Z^1(\Mod_{g,b}; H_g^{b-1})$ as claimed.
        \end{proof}

        Here we prove Theorem \ref{T:MainMCGCohomology}.

        \begin{proof}[Proof of Theorem \ref{T:MainMCGCohomology}]
            Consider the Birman exact sequence
            \[
            1 \to \pi_1 \Sigma_{g,n} \to \Mod_{g,n+1} \to \Mod_{g,n} \to 1
            \]
            induced by forgetting the $n+1^{st}$ puncture. We consider the associated inflation-restriction sequence, taking coefficients in $H_g^n$:
            \[
            0 \to H^1(\Mod_{g,n}; H_g^n) \to H^1(\Mod_{g,n+1}; H_g^n) \to H^1(\Sigma_{g,n}; H_g^n)^{\Mod_{g,n}} \to \dots
            \]
            As $n \ge 2$, Proposition \ref{P:puncturevanish} implies the vanishing of the first term. Corollary \ref{C:EquivariantH} asserts that 
            \[
            H^1(\Sigma_{g,n}; H_g^n)^{\Mod_{g,n}} \cong \Hom(H_g^n, H_g^n)^{\Mod_{g,n}} \cong A.
            \]
            Thus, this sequence simplifies as shown:
            \[
            0 \to H^1(\Mod_{g,n+1}; H_g^n) \to A \to \dots
            \]
            Thus to prove Theorem \ref{T:MainMCGCohomology}, we need only show that the class $\overline{C(\xi_0)} \in Z^1(\Mod_{g,n+1}; H_g^n)$ (as constructed in Lemma \ref{L:barCpsi}) has nontrivial restriction to $H^1(\Sigma_{g,n}; H_g^n)^{\Mod_{g,n}}$.

            This can be computed directly. Let $\gamma \in \pi_1 \Sigma_{g,n}$ be given. As an element of $\Mod_{g,n+1}$, this is realized by pushing the $n+1^{st}$ marked point along the loop $\gamma$. As a mapping class, this is given as $T_{\gamma_R} T_{\gamma_L}^{-1}$, where $\gamma_L, \gamma_R$ are simple closed curves on $\Sigma_{g,n+1}$ obtained by pushing $\gamma$ off the last marked point to the left/right, respectively. 
            Combining twist-linearity and homological coherence, the effect on a relative class $\alpha \in H_1(\Sigma_{g,n},P)$ is seen to be as 
            \[
            \gamma \cdot \alpha = \alpha + \chi(\Sigma_{g,n+1})\pair{[\gamma], \alpha}[\gamma].
            \]
            Unraveling definitions, this shows that $\overline{C(\xi_0)}$ is sent to the nonzero class $\chi(\Sigma_{g,n+1})\Id \in  \Hom(H_g^n, H_g^n)^{\Mod_{g,n}}$.
        \end{proof}

    \subsection{Proof of Proposition \ref{P:ArcFramedMCG}}\label{SS:Aarcframed}

Let $\eta, \xi_0$ be  non-vanishing vector fields on the punctured surface $\Sigma_{g,n+1}$, where $\xi_0$ has zero winding number around each of the first $n$ marked points; $\eta$ can be arbitrary. Our goal here is to identify conditions under which the cocycle $\overline{C(\xi_0)}$ generating $H^1(\Mod_{g,n+1}; H_g^n)$ vanishes when restricted to a subgroup of the framed mapping class group $\Mod_{g,n+1}[\eta]$.

In this setting one must work with both punctures and marked points. The $n+1$ marked points $\{p_1, \dots, p_{n+1}\}$ on $\Sigma_{g,n+1}$ will be denoted $Q$, and the set of the first $n$ of these will be denoted $Q'$. As before, we will let $\Delta_i$ denote a loop around $p_i$, oriented with $p_i$ to the right. There are natural identifications
\[
H_1(\Sigma_g \setminus Q) \cong H_g^{n+1} \quad \mbox{and} \quad H_1(\Sigma_g \setminus Q') \cong H_g^n.
\]
We will also make use of the homological boundary maps
\[
\delta: H_1(\Sigma_g,Q) \to \wt H_0(Q) \quad \mbox{and} \quad \delta: H_1(\Sigma_g, Q') \to \wt H_0(Q').
\]

In preparation for Proposition \ref{P:ArcFramedMCG}, define the element 
\[
D_\eta = \sum_{i = 1}^{n+1} (W_{\xi_0}(\Delta_i) - W_{\eta}(\Delta_i)) p_i \in \wt H_0(Q).
\]
We also let $T = \{\alpha_1, \dots, \alpha_{n-1}\}$ be a set of $n-1$ arcs on the marked surface $(\Sigma_g, Q')$ such that their images in $\wt H_0(Q')$ forms a basis; we call such a set a {\em basis of arcs}. Since each $\alpha_i$ is based at $Q'$, the winding number $W_{\xi_0}(\alpha_i) \in \bZ$ is defined.

\begin{prop}\label{P:ArcFramedMCG}
 Let $G \leq \Mod_{g,n+1}[\eta]$. The restriction of $\overline{C(\xi_0)} \in H^1(\Mod_{g,n+1}; H_{g}^n)$ to $G$ is a coboundary if and only if
 \begin{enumerate}[label=(\alph*)]
 \item There is an element $\gamma \in H_1(\Sigma_g, Q)$ with boundary $\delta(\gamma) = D_\eta \in \wt H_0(Q)$ such that $g \gamma = \gamma$ for all $g \in G$,
 \item There is a non-vanishing vector field $\xi_1$ with $B_{\xi_1} = B_{\xi_0}$ and a basis of arcs $T$ such that $W_{\xi_1}(g(\alpha_i)) = W_{\xi_1}(\alpha_i)$ for all $\alpha_i \in T$ and all $g \in G$.
 \end{enumerate}
\end{prop}

Combining this with Theorem \ref{T:MainMCGCohomology}, we obtain the culminating result of the Appendix.

\begin{cor} \label{C:zerorestrict}
    Let $\Sigma_g$ be a surface of genus $g \ge 2$, and let $Q = \{p_1, \dots, p_{n+1}\}$ be a set of marked points. Let $Q' = Q \setminus \{p_{n+1}\}$. 
    Let $G \le \Mod(\Sigma_g \setminus Q)$ be a subgroup. Suppose $G$ satisfies the following conditions:
    \begin{enumerate}
        \item There is a framing $\eta$ of $\Sigma_g \setminus Q$ for which $G \le \Mod(\Sigma_g \setminus Q)[\eta]$.
        \item There is a basis $\tau$ of arcs in $\wt{H}^0(Q)$  such that $g[\alpha_i] = [\alpha_i]$ in $H_1(\Sigma_g, Q)$ for all $g \in G$ and $\alpha_i \in \tau$.
        \item There is a non-vanishing vector field $\xi_1$ on $\Sigma_g \setminus Q$ with zero winding number about each $p_i \in Q'$ such that $W_{\xi_1}(g(\beta_i)) = W_{\xi_1}(\beta_i)$ for all $g \in G$ and for all arcs $\beta_i$ forming a basis of $\wt{H}_0(Q')$. 
    \end{enumerate}
    Then the restriction map $H^1(\Mod_{g,n+1}; H_1(\Sigma_g\setminus Q')) \to H^1(G; H_1(\Sigma_g\setminus Q'))$ is zero.
\end{cor}

Let $H$ denote $H_1(\Sigma_g)$. As always, $H_g^n$ denotes $H_1(\Sigma_{g,n}) \cong H_1(\Sigma_g^n) = H_1(\Sigma_g\setminus Q')$. The proof of Proposition \ref{P:ArcFramedMCG} is based around an analysis of the sequence of coefficients
\begin{equation}\label{E:SES-PuncturedH1}
 0 \ra K_g^n \ra H_g^n \xra{\pi} H \ra 0.   
\end{equation} 

For any subgroup $G \le \Mod_{g,n+1}$, this induces the exact sequence
\begin{align}\label{E:Gexact}
    H^1(G; K_g^n) \to H^1(G; H_g^n) \to H^1(G;H).
\end{align}

Condition (a) of Proposition \ref{P:ArcFramedMCG} expresses the condition for the cocycle $\overline{C(\xi_0)}$ to project to a coboundary in $Z^1(G,H)$, and condition (b) then expresses the condition under which the induced element of $H^1(G;K_g^n)$ is trivial. We will analyze each of these conditions in turn.

\begin{lem}\label{L:MoreGeneralCoboundary}
Let $\eta$ be a framing of $\Sigma\setminus Q$ and let $G \leq \mathrm{Mod}_{g,n+1}[\eta]$ be a subgroup of the associated framed mapping class group. The following are equivalent:
\begin{enumerate}
    \item $\pi_*(\overline{C(\xi_0)}) \in H^1(G; H) = 0$.
    \item There is a $G$-invariant element $\gamma \in H_1(\Sigma_g, Q)$ with $\delta(\gamma) = D_\eta$. 
\end{enumerate}
\end{lem}
\begin{proof}
By definition, the class of the restriction $\pi_*(\overline{C(\xi_0)})$ to $G$ is zero in $H^1(G;H)$ if and only if the cocycle $\pi_*(\overline{C(\xi_0)})$ is a coboundary. We give an explicit description of what this means. Viewing $\overline{C(\xi_0)}: G \to H_g^n$ as a cocycle, we first interpret $H_g^n = H_1(\Sigma_g\setminus Q')$ as dual (via Lefschetz duality) to $H_1(\Sigma_g,Q')$. Under this identification, the projection $\pi: H_g^n \to H$ corresponds to a restriction of the domain of $\overline{C(\xi_0)}(g) \in H_1(\Sigma_g,Q')^*$ to the subspace $H$.

Unwinding the definitions, if this projection is a coboundary, then there exists $y \in H$ such that for all $x \in H$ and all $g \in G$,
\begin{align} \label{E:coboundary1}
    W_{\xi_0}(g(\hat x)) - W_{\xi_0}(\hat x) = \pair{(g_*-1)y,x}, 
\end{align}

where, as always, $\hat x$ denotes the Johnson lift of a simple closed curve on $\Sigma_{g,n}$ with homology class $x \in H$. 

Recall from Proposition \ref{prop:framingchar} that if $\psi, \psi'$ are any two non-vanishing vector fields on $\Sigma_{g,n+1}$, then there exists $v \in H^1(\Sigma_{g,n+1})$ such that
\[
W_{\psi'} = W_{\psi} + p^*v
\]

Applying this to $\xi_0$ and $\eta$, there is $v \in H^1(\Sigma_{g,n+1})$ for which $W_{\xi_0} = W_{\eta} + p^* v$. Note that if $\Delta_i$ is the loop around $p_i \in Q$, then 
\[
v(\Delta_i) = W_{\xi_0}(\hat\Delta_i) - W_{\eta}(\hat \Delta_i).
\]
Under the Lefschetz duality isomorphism $H^1(\Sigma_g\setminus Q) \cong H_1(\Sigma_g, Q)$, this implies that $\delta(v) = D_\eta$.

Returning to \eqref{E:coboundary1}, we can make two adjustments. First, treating the winding number function as a cohomology class, the left-hand term can be combined into $W_{\xi_0}((g_*-1)\hat x)$. Secondly, expanding $W_{\xi_0} = W_{\eta} + p^* v$, we observe that since $g \in G \le \Mod_{g,n+1}[\eta]$, the term $W_{\eta}((g_*-1)\hat x) = 0$. Thus \eqref{E:coboundary1} reduces to
\[
(g_*-1)v(x) = \pair{(g_*-1)y,x}.
\]
Since $y \in H = \ker \delta \le H_{g,n+1}$, the element $\gamma  = v - y \in H^1(\Sigma_{g,n+1}) \cong H_1(\Sigma_g,Q)$ is the desired $G$-invariant class with boundary $\delta(\gamma) = D_\eta$ that exists if and only if the restriction of $\pi_*(\overline{C(\xi_0)})$ to $H^1(G;H)$ is zero.
\end{proof}

\begin{proof}[Proof of Proposition \ref{P:ArcFramedMCG}:] Recall the exact sequence of \eqref{E:Gexact}:   
\[ 
H^1(G;K_g^n) \ra H^1(G;H_g^n) \xra{\pi_*} H^1(G;H). 
\]
By Lemma \ref{L:MoreGeneralCoboundary}, $\pi_*(\overline{C(\xi_0)})$ is a coboundary if and only if $G$ fixes some $\gamma \in H_1(\Sigma_g,Q)$ with $\delta(\gamma) = D_\eta$. Supposing this is the case, it follows by exactness of \eqref{E:Gexact} that there is an expression
\[
\overline{C(\xi_0)}(g) = \phi(g) + (g_*-1)y
\]
for some $y \in H_g^n$ and some homomorphism $\phi: G \to K_g^n$ (and all $g \in G$). 

Lifting $y \in H_g^n$ to $\wt y \in H_g^{n+1}$ and $\overline{C(\xi_0)} \in Z^1(\Mod_{g,n+1}, H_g^n)$ to $C(\xi_0) \in Z^1(\Mod_{g}^{n+1}, H_g^{n+1})$, it follows from Lemma \ref{L:cohomVF} that there is a non-vanishing vector field $\xi_1$ with $B_{\xi_1} = B_{\xi_0}$ such that
\[
\phi(g) = \overline{C(\xi_0)}(g) - (g_*-1)y = \overline{C(\xi_1)}(g)
\]

Thus $\overline{C(\xi_0)}$ (equivalently, $\overline{C(\xi_1)}$) is a coboundary if and only if $\phi$ is the zero homomorphism. Under the Lefschetz duality isomorphism $H_1(\Sigma_g \setminus Q') \cong H_1(\Sigma_g, Q')^*$, the subspace $K_g^n \le H_1(\Sigma_g \setminus Q')$ is dual to the quotient space $\wt H_0(Q')$ of $H_1(\Sigma_g, Q')$. Thus $\phi(g) \equiv 0$ if and only if $\pair{\phi(g),x} = 0$ for all $x \in H_1(\Sigma_g, Q')$, if and only if $\pair{\phi(g),\alpha_i} = 0$ for some set of $\alpha_i$ whose images span $\wt H_0(Q')$, i.e. for some basis of arcs. As $\phi(g) = \overline{C(\xi_1)}(g)$, this latter condition is seen to be the condition (b) of Proposition \ref{P:ArcFramedMCG}.
\end{proof}

\bibliography{mybib}{}
\bibliographystyle{amsalpha}
\end{document}